\documentclass[pdftex]{amsart}
\usepackage{amsmath, amssymb, mathrsfs, stmaryrd, amsrefs, comment}
\usepackage[all]{xy}
\allowdisplaybreaks
\usepackage{amsthm}
\theoremstyle{definition}
\newtheorem{definition}{Definition}[subsection]
\theoremstyle{plain}
\newtheorem{theorem}[definition]{Theorem}
\newtheorem{proposition}[definition]{Proposition}
\newtheorem{lemma}[definition]{Lemma}
\newtheorem{corollary}[definition]{Corollary}
\newtheorem{fact}[definition]{Fact}
\theoremstyle{remark}
\newtheorem*{remark}{Remark}
\newtheorem*{notation}{Notation}
\newtheorem{example}[definition]{Example}
\usepackage[svgnames]{xcolor}
\usepackage{hyperref}
\hypersetup{
  colorlinks=true,
  linkcolor=Blue,
  citecolor=Green,
  pdftitle={},
  pdfpagemode=FullScreen
}
\newcommand\e{{\boldsymbol e}}

\renewcommand\H{{\mathcal H}}
\newcommand\I{{\mathcal I}}
\renewcommand\L{{\mathcal L}}
\newcommand\K{{\mathcal K}}
\renewcommand\O{{\mathcal O}}

\newcommand\N{{\mathbb N}}
\newcommand\Z{{\mathbb Z}}
\newcommand\C{{\mathbb C}}

\newcommand\X{{\mathbb X}}

\newcommand\id{{\mathrm{id}}}

\newcommand\cspan{\mathop{\overline{\mathrm{span}}}}
\newcommand\asg[1]{\langle{#1}\rangle}
\newcommand\casg[1]{\overline{\asg{#1}}}

\newcommand\supp{\mathop{\mathsf{supp}}}
\newcommand\defiff{\overset{\mathrm{def}}{\iff}}
\newcommand\rel[1]{\llbracket#1\rrbracket}
\newcommand\ent{\mathop{\mathrm{ent}}\nolimits}

\newcommand\set[2]{\mathopen\lbrace#1\mathbin\vert#2\mathclose\rbrace}

\newcommand\inn[2]{\langle#1\vert#2\rangle}
\newcommand\INN[3]{\mathopen{#1\langle}#2\mathbin{#1\vert}#3\mathclose{#1\rangle}}
\newcommand\abs[1]{\vert#1\vert}

\newcommand\norm[1]{\Vert#1\Vert}
\newcommand\NORM[2]{\mathopen{#1\Vert}#2\mathclose{#1\Vert}}
\begin{document}
\title
[Cartan subalgebras of Kajiwara--Watatani algebras]
{Cartan subalgebras of C*-algebras associated\\with iterated function systems}
\author[K. Ito]{Kei Ito}
\address{Graduate School of Mathematical Sciences, The University of Tokyo, 3-8-1 Komaba, Tokyo, 153-8914, Japan}
\email{itokei@ms.u-tokyo.ac.jp}
\maketitle

\begin{abstract}
  Let $X$ be a compact Hausdorff space, and let $\gamma$ be an iterated function system on $X$.
  Kajiwara and Watatani showed that if $\gamma$ is self-similar and satisfies the open set condition and some additional technical conditions, $C(X)$ is a maximal abelian subalgebra of the Kajiwara--Watatani algebra $\O_\gamma$ associated with $\gamma$.
  In this paper, we extend their results by providing sufficient conditions under which $C(X)$ becomes a Cartan subalgebra of $\O_\gamma$.
  Additionally, we present sufficient conditions under which $C(X)$ fails to be a masa of $\O_\gamma$.
\end{abstract}

\section{Introduction}
\subsection{Historical background}

In this paper, we study a Cartan subalgebra in a C*-algebra associated with an iterated function system.

A C*-algebra associated with a self-similar map, which is an example of iterated function systems, was introduced by Kajiwara and Watatani \cite{KW06}. 
Their aim was to improve the framework created by Deaconu and Muhly, which was based on branched coverings.

Deaconu and Muhly \cite{DM} used Renault's groupoid method \cite{R} to define a C*-algebra for a branched covering $\sigma\colon X\to X$. 
This algebra is the groupoid C*-algebra of the topological groupoid $\Gamma(X,T)$, where $T$ is the locally homeomorphic part of $\sigma$. 
Kajiwara and Watatani noticed that this method ignores branch points because it only uses the locally homeomorphic sections. 
To address this, they used the Cuntz--Pimsner construction \cite{P} to create new C*-algebras that include branch point information. 
These algebras apply to both complex dynamical systems \cite{KW05} and self-similar maps \cite{KW06}.

The C*-algebra by Deaconu and Muhly for a branched covering $\sigma\colon X\to X$ naturally has a Cartan subalgebra, which is isomorphic to $C_0(X)$, the algebra of continuous functions on $X$ that vanish at infinity. 
Similarly, Kajiwara and Watatani \cite{KW17} showed that their algebra for a dynamical system on $X$ has a maximal abelian subalgebra isomorphic to $C_0(X)$. 
However, it was still unclear if this subalgebra was a Cartan subalgebra. 
For complex dynamical systems, this problem was resolved in \cite{I}, where the necessary and sufficient conditions for $C_0(X)$ to be a Cartan subalgebra were given: if $R$ is a rational function and $X$ is its Julia set, Fatou set, or the Riemann sphere, then $C_0(X)$ is a Cartan subalgebra of the Kajiwara--Watatani algebra $\O_R(X)$ if and only if $R$ has no branch points on $X$. 

In this paper, we study the conditions under which $C_0(X)$ becomes a Cartan subalgebra for iterated function systems.

The notion of Cartan subalgebras was introduced by Kumjian \cite{Ku} and further developed by Renault \cite{R1980,R}. 
According to Kumjian \cite{Ku}, the Cartan subalgebra was introduced with the belief that understanding how abelian subalgebras embed in a C*-algebra illuminates its structure.
Renault's theory shows that for a separable C*-algebra $A$ with a Cartan subalgebra $B$, there is a topological groupoid $G$ with a unit space $G^{(0)}$ and twist data $\Sigma$ such that $A$ is isomorphic to the groupoid C*-algebra $C^*(G,\Sigma)$, and $C_0(G^{(0)})$ maps onto $B$. 
Moreover, this twisted groupoid $(G,\Sigma)$ is isomorphic to the Wyle twist $(G(B),\Sigma(B))$; and hence unique up to isomorphism.

Let $K$ be a compact metric space, and let $\gamma = \{\gamma_1,\ldots,\gamma_N\}$ be an iterated function system on $K$, i.e., a finite family of proper contractions from $K$ into itself. We say that $\gamma$ is a self-similar map on $K$ if $K = \gamma_1(K)\cup\cdots\cup\gamma_N(K)$. 
Let $\O_\gamma$ denote the Kajiwara--Watatani algebra for an iterated function system $\gamma$. 
The open set condition for an iterated function system $\gamma$ is the condition that there exists an open set $U$ in $K$ such that the images $\gamma_1(U),\ldots,\gamma_N(U)$ are pairwise disjoint and contained in $U$. 
Kajiwara and Watatani showed in \cite{KW17} that $C(K)$ is a maximal abelian subalgebra (masa) of $\O_\gamma$ if $\gamma$ is a self-similar map and satisfies the open set condition and some additional technical conditions. 
However, whether this masa $C(K)$ is a Cartan subalgebra remains an open question.

In this paper, we study the conditions under which $C(K)$ becomes a masa and the conditions under which it becomes a Cartan subalgebra. To do this, we present an explicit universal covariant representation of Kajiwara--Watatani algebras on a Hilbert space, avoiding the use of quotient constructions like the GNS representation. This representation helps us to further analyze Kajiwara--Watatani algebras.

\subsection{Main theorems}

We refer to Section 3 and 4 for the definition of the terminology used below.
Let $X$ be a compact Hausdorff space, $\gamma = \{\gamma_1,\ldots,\gamma_N\}$ be an iterated function system on $X$, and $\X$ be a subset of $X$.
Let $G$ be the graph of $\gamma$ and the pair of $A = C(X)$ and $E = C(G)$ be the Kajiwara--Watatani correspondence for $(X,\gamma)$.
We assume the following conditions:
\begin{itemize}
  \item The set $\X$ is dense in $X$ and completely invariant under $\gamma$.
  \item The images $\gamma_1(\X),\ldots,\gamma_N(\X)$ are pairwise disjoint.
  \item The restriction $\gamma_k|_\X$ is injective for all $k=1,\ldots,N$.
\end{itemize}
Then, we can construct a manageable universal covariant representation $\hat\rho = (\hat\rho_A, \hat\rho_E)$ of the Kajiwara--Watatani correspondence for $(X,\gamma)$ (Theorem~\ref{5.2}).
So we can regard the Kajiwara--Watatani algebra $\O_\gamma$ for $\gamma$ as $C^*(\hat\rho)$.

We assume the following additional conditions:
\begin{itemize}
  \item The maps $\gamma_1,\ldots,\gamma_N$ are embeddings.
%
%

%
  \item The iterated function system $\gamma$ satisfies the open set condition. That is, there exists a dense open set $U$ in $X$ such that the images $\gamma_1(U),\ldots,\gamma_N(U)$ are pairwise disjoint open set in $X$ and contained in $U$.
  \item The set $\X$ is contained in $U$.
  \item The set $\X$ is comeager in $X$.
\end{itemize}
Then, we can construct a masa $\hat D$ in $C^*(\hat\rho)$ that contains $\hat\rho_A(A)$ (Theorem~\ref{5.5.1}).
So $\hat\rho_A(A)$ must coincides with $\hat D$ for $\hat\rho_A(A)\cong C(X)$ to be a masa of $C^*(\hat\rho)\cong\O_\gamma$.
The system $\gamma$ must be essentially free for $\hat\rho_A(A)\cong C(X)$ to be a masa of $C^*(\hat\rho)\cong\O_\gamma$.
%
%
\begin{theorem}[Theorem~\ref{5.5.5}]\label{1.2.6}
  If $\gamma$ is not essentially free, then $\hat\rho_A(A)$ is properly contained in $\hat D$ and hence is not a masa of $C^*(\hat\rho)$.
\end{theorem}

\begin{theorem}[Theorem~\ref{5.5.5} and Theorem~\ref{5.5.7}]\label{1.2.7}
  Assume the following conditions:
  \begin{itemize}
    \item the iterated function system $\gamma$ is essentially free; 
    \item the set $\gamma^n(X)$ is clopen in $X$ for all $n\in\N$; 
    \item there exists a continuous map $\sigma$ from $\gamma(X)$ to $X$ such that $\sigma\circ\gamma_k = \id_X$ holds for all $k\in\{1,\ldots,N\}$.
  \end{itemize}
  Then, $\hat\rho_A(A)$ coincides with $\hat D$ and hence is a masa of $C^*(\hat\rho)$.
\end{theorem}


Assuming the three conditions stated in Theorem~\ref{1.2.7}, consider whether the masa $\hat D = \hat\rho_A(A)$ is a Cartan subalgebra or not in the following theorem:
%

\begin{theorem}[Theorem~\ref{5.9.8} and Theorem~\ref{5.10.2}]\label{1.2.10}
  Assume the three conditions stated in Theorem~\ref{1.2.7} and assume that the space $X$ is first-countable.
  Then, the following are equivalent:
  \begin{enumerate}
    \item $\gamma$ satisfies the graph separation condition, i.e., for all distinct indices $j$ and $k$ and all points $x \in X$, one has $\gamma_j(x) \neq \gamma_k(x)$.
    \item $\hat D$ is a Cartan subalgebra of $C^*(\hat\rho)$.
  \end{enumerate}
\end{theorem}

\subsection{Structure of the Paper}

The paper is organized as follows:

\begin{itemize}
  \item In \textbf{Section 2}, we introduce the notion of $X$-squared matrices and review Cuntz--Pimsner--Katsura algebras. In this paper, we represent the Kajiwara--Watatani algebra using $X$-squared matrices, and by analyzing the $X$-squared matrices, we analyze the Kajiwara--Watatani algebra.
  \item \textbf{Section 3} provides a review and preparation of iterated function systems. Proposition~\ref{3.15} and Proposition~\ref{3.21} in this section play an important role in clarifying the structure of the Kajiwara--Watatani algebra.
  \item \textbf{Section 4} provides a review of the Kajiwara--Watatani algebra.
  \item \textbf{Section 5} is the main body of the paper. In this section, we analyze the Kajiwara--Watatani algebra and provide conditions under which $C(X)$ becomes the masa or Cartan subalgebra of $\O_\gamma$.
  \item In \textbf{Section 6}, we present examples in which $C(X)$ is not the masa of $\O_\gamma$, in which it is the masa but not the Cartan subalgebra, and in which it is the Cartan subalgebra.
\end{itemize}

\section{Preparations for Operator Algebras}
\subsection{Adjointable operators between right Hilbert modules}

In this subsection, we review adjointable operators between right Hilbert modules. See \cite{B} Part II, Section 7 for more detail.

Let $A$ be a C*-algebra and $E$ and $F$ be right Hilbert $A$-modules.
An \emph{adjoint} of a complex-linear map $T\colon E\to F$ is a complex-linear map $S\colon F\to E$ satisfying the equation $\inn{e}{Sf}_E=\inn{Te}{f}_F$ for any $e\in E$ and $f\in F$.
A complex-linear map is called \emph{adjointable} if it has an adjoint.
An adjointable map is bounded and $A$-linear.
Let $\L(E,F)$ denote the set of all adjointable operators from $E$ to $F$.
Endowed with operator norm, the set $\L(E,F)$ is a Banach space.

Let $\Theta(E,F):=\set{\theta_{f,e}}{e\in E\text{ and }f\in F}\subseteq\L(E,F)$, where $\theta_{f,e}$ is defined as the operator which maps $x\in E$ to $f\inn{e}{x}_E$.
This is closed under complex-scalar multiplication.
Let $\K(E,F)$ denote the closed subspace of $\L(E,F)$ generated by $\Theta(E,F)$.
Then, $\L(E):=\L(E,E)$ is a C*-algebra and $\K(E):=\K(E,E)$ is its closed ideal.

For later use, let $K(E)$ denote the algebraic additive subgroup of $\L(E)$ generated by $\Theta(E):=\Theta(E,E)$.
Since $\Theta(E)$ is closed under complex-scalar multiplication, $K(E)$ is a linear subspace of $\L(E)$.
Therefore, the closure of $K(E)$ is $\K(E)$.

\subsection{$X$-squared matrices}

Let $X$ be a set.
Let $\H_X$ denote the Hilbert space $\bigoplus_X\C$ with the standard basis $\{\e_x\}_{x\in X}$.
Let $M_X$ denote the C*-algebra of all bounded linear operators on $\H_X$.
We call an element of $M_X$ an \emph{$X$-squared matrix}.
For $x,y\in X$, \emph{the $(x,y)$-entry} of $\alpha\in M_X$ means the value $\inn{\e_x}{\alpha\e_y}$.

\begin{definition}\label{2.1}
  Let $R$ be a binary relation on $X$.
  \begin{enumerate}
    \item We say that $R$ is \emph{the identity} if $xRy$ implies $x=y$ for all $x,y\in X$.
    \item We say that $R$ is \emph{one-to-one} if $xRy$ and $xRz$ imply $y=z$ and also $xRz$ and $yRz$ imply $x=y$ for all $x,y,z\in X$.
  \end{enumerate}
\end{definition}

\begin{definition}\label{2.2}
  For an $X$-squared matrix $\alpha$, we define a binary relation $\rel\alpha$ on $X$ by the following: for each $x,y\in X$,
  \begin{equation*}
    x\rel\alpha y\defiff\text{the $(x,y)$-entry of $\alpha$ does not vanish, i.e., $\inn{\e_x}{\alpha\e_y}\neq 0$}.
  \end{equation*}
\end{definition}

\begin{definition}\label{2.3}\ 
  \begin{enumerate}
    \item An $X$-squared matrix $\alpha$ is \emph{diagonal} if the binary relation $\rel\alpha$ is the identity.
    Let $D_X$ denote the set of all diagonal $X$-squared matrices.
    \item An $X$-squared matrix $\alpha$ is \emph{quasi-monomial} if the binary relation $\rel\alpha$ is one-to-one.
    Let $QM_X$ denote the set of all quasi-monomial $X$-squared matrices.
  \end{enumerate}
\end{definition}

\begin{fact}[\cite{I}, Proposition 2.4]\label{2.4}
  Let $\alpha$ and $\beta$ be $X$-squared matrices, $c$ be a complex number, and $x$ and $y$ be points in $X$.
  Then, the following assertions hold:
  \begin{enumerate}
    \item If $x\rel{\alpha+\beta}y$, then $x\rel\alpha y$ or $x\rel\beta y$.
    \item If $x\rel{\alpha c}y$, then $x\rel\alpha y$.
    \item If $x\rel{\alpha\beta}y$, then $x\rel\alpha\rel\beta z$.
    \item If $x\rel{\alpha^*}y$, then $y\rel\alpha x$.
  \end{enumerate}
\end{fact}

\begin{fact}[\cite{I}, Proposition 2.6]\label{2.5}
  The complex-linear operator $\delta\colon M_X\to D_X$ defined below is a faithful conditional expectation:
  \begin{equation*}
    \delta(\alpha)\e_x:=\e_x\inn{\e_x}{\alpha\e_x}\quad\text{for each }\alpha\in M_X\text{ and }x\in X.
  \end{equation*}
\end{fact}

\subsection{Cuntz--Pimsner--Katsura algebras}

In this subsection, we review Cuntz--Pimsner--Katsura algebras.
See \cite{K} for more details.

A \emph{C*-correspondence} is a pair of a C*-algebra $A$ and a right Hilbert $A$-module $E$ endowed with the left $A$-scalar multiplication induced by a $*$-homomorphism from $A$ to $\L(E)$.

Let $(A,E)$ be a C*-correspondence, $B$ be C*-algebras, and $\rho_A\colon A\to B$ and $\rho_E\colon E\to B$ be maps.
The pair $\rho=(\rho_A,\rho_E)$ is called a \emph{representation} of $(A,E)$ if $\rho_A$ is a $*$-homomorphism, $\rho_E$ is a complex-linear map, and moreover the following two conditions hold.
\begin{enumerate}
  \item The pair $\rho$ preserves inner product, i.e., $\rho_A(\inn{f}{g})=\rho_E(f)^*\rho_E(g)$ for all $f,g\in E$.
  \item The pair $\rho$ preserves left $A$-scalar multiplication, i.e., $\rho_E(af)=\rho_A(a)\rho_E(f)$ for all $a\in A$ and $f\in E$.
\end{enumerate}
Then, the next holds automatically.
See \cite{I}, Fact 2.7 for the proof of this fact.
\begin{enumerate}
  \setcounter{enumi}{2}
  \item The pair $\rho$ preserves right $A$-scalar multiplication, i.e., $\rho_E(fa)=\rho_E(f)\rho_A(a)$ for all $f\in E$ and $a\in A$.
\end{enumerate}
For a representation $\rho=(\rho_A,\rho_E)$, let $C^*(\rho)$ denote the C*-subalgebra of $B$ generated by the set $\rho_A(A)\cup\rho_E(E)$.

A representation $\rho$ is \emph{injective} if $\rho_A$ is injective.
In this case, $\rho_E$ is also injective automatically.

Let $(A,E)$ be a C*-correspondence and $\rho=(\rho_A,\rho_E)$ be a representation of $E$.
We define two $*$-homomorphisms $\varphi$ and $\psi_\rho$ as follows.
\begin{alignat*}{3}
  &\varphi\colon A\to\L(E),
  &\quad&\varphi(a)(f):=af
  &\quad&\text{for $a\in A$, $f\in E$}.
  \\
  &\psi_\rho\colon\K(E)\to C^*(\rho),
  &\quad&\psi_\rho(\theta_{f,g}):=\rho_E(f)\rho_E(g)^*
  &\quad&\text{for $f,g\in E$}.
\end{alignat*}
We call $\varphi$ \emph{the natural left action of $A$ on $E$}, and $\psi_\rho$ \emph{the representation of $\K(E)$ induced by $\rho$}.
Let $\I_E:=\set{a\in A}{\varphi(a)\in\K(E)\text{ and }ab=0\text{ for all }b\in\ker\varphi}$.
This is a closed ideal of $A$ and is called a \emph{covariant ideal} of $(A,E)$.
We call $\rho$ \emph{covariant} if $\rho_A(a)=\psi_\rho(\varphi(a))$ holds for all $a\in\I_E$.
There exist covariant representations, and the universal one exists.
\emph{The Cuntz--Pimsner--Katsura algebra} associated with $(A,E)$ is the C*-algebra induced by the universal covariant representation of $(A,E)$ and denoted by $\O_E$.

\begin{fact}[\cite{I}, Proposition 2.9]\label{2.6}
  The equation $\psi_\rho(\varphi(a))\rho_E(f)=\rho_A(a)\rho_E(f)$ holds for all $a\in\I_E$ and $f\in E$.
\end{fact}

\begin{proposition}\label{2.7}
  The equation $\psi_\rho(\theta)\rho_A(b)=0$ holds for all $\theta\in\K(E)$ and $b\in\ker\varphi$.
\end{proposition}

\begin{proof}
  The equation holds for all $\theta\in\Theta(E)$ since we have the following for any $f,g\in E$ and $b\in\ker\varphi$:
  \begin{gather*}
    \psi_\rho(\theta_{f,g})\rho_A(b)
    =\rho_E(f)\rho_E(g)^*\rho_A(b)
    =\rho_E(f)(\rho_A(b)^*\rho_E(g))^*\\
    =\rho_E(f)\rho_E(b^*g)^*
    =\rho_E(f)\rho_E(\varphi(b^*)(g))^*
    =\rho_E(f)\rho_E(0)^*
    =0.
  \end{gather*}
  By the linearity, it also holds for all $\theta\in K(E)$.
  Since the linear operator $\theta\mapsto\psi_\rho(\theta)\rho_A(b)$ is bounded, the equation holds for all $\theta\in\K(E)$.
\end{proof}

\begin{fact}[\cite{I}, Proposition 2.12]\label{2.8}
  The equation $\lim_\lambda(f\cdot a_\lambda)=f$ holds for all approximate unit $\{a_\lambda\}$ of $A$ and $f\in E$.
\end{fact}

\begin{definition}[\cite{K}, Definition 5.6]\label{2.9}
  A representation $\rho=(\rho_A,\rho_E)$ of a C*-correspondence $(A,E)$ is said to \emph{admit a gauge action} if for each complex number $c$ with $\abs{c}=1$, there exists a $*$-homomorphism $\Phi_c\colon C^*(\rho)\to C^*(\rho)$ such that $\Phi_c(\rho_A(a))=\rho_A(a)$ and $\Phi_c(\rho_E(f))=c\rho_E(f)$ for all $a\in A$ and $f\in E$.
\end{definition}

\begin{fact}[\cite{K}, Theorem 6.4]\label{2.10}
  For a covariant representation $\rho$ of a C*-correspondence $(A,E)$, the $*$-homomorphism $\O_E\to C^*(\rho)$ is an isomorphism if and only if $\rho$ is injective and admits a gauge action.
\end{fact}

\subsection{The group C*-algebra $C^*(\Z)$}
\label{subsection 2.4}

For each $n\in\Z$, we define a $\Z$-squared matrix $\upsilon_n$ by $\upsilon_n\e_k:=\e_{n+k}$ for all $k\in\Z$.
The group C*-algebra $C^*(\Z)$ is the C*-subalgebra of $M_\Z$ generated by $\set{\upsilon_n}{n\in\Z}$.

\begin{fact}[\cite{I}, Proposition~2.15]\label{2.11}
  Let $c$ be a complex number with absolute value~$1$.
  Then, there exists a unique $*$-homomorphism $\Phi_c\colon C^*(\Z)\to C^*(\Z)$ satisfying the equation $\Phi_c(\upsilon_n)=\upsilon_n c^n$ for all $n\in\Z$.
\end{fact}

\begin{proposition}\label{2.12}
  Let $(A,E)$ be a C*-correspondence, $B$ be a C*-algebra, and $\rho=(\rho_A,\rho_E)$ be an injective covariant representation of $(A,E)$ in $B$.
  Define two maps $\hat\rho_A\colon A\to B\otimes C^*(\Z)$ and $\hat\rho_E\colon E\to B\otimes C^*(\Z)$ by 
  \begin{alignat*}{3}
    \hat\rho_A(a)&:=\rho_A(a)\otimes\upsilon_0&&\quad\text{for each $a\in A$; and}\\
    \hat\rho_E(f)&:=\rho_E(f)\otimes\upsilon_1&&\quad\text{for each $f\in E$.}
  \end{alignat*}
  Then, $\hat\rho=(\hat\rho_A,\hat\rho_E)$ is a universal covariant representation of $(A,E)$.
\end{proposition}

\begin{proof}
  We show that $\hat\rho$ is a representation.
  It is easy to verify that $\hat\rho_A$ is a $*$-homomorphism and that $\hat\rho_E$ is a complex-linear map.
  The pair $\hat\rho$ preverves inner product since the following holds for any $f,g\in E$:
  \begin{gather*}
    \hat\rho_E(f)^*\hat\rho_E(g)
    =(\rho_E(f)\otimes\upsilon_1)^*(\rho_E(g)\otimes\upsilon_1)
    =(\rho_E(f)^*\otimes\upsilon_{-1})(\rho_E(g)\otimes\upsilon_1)\\
    =(\rho_E(f)^*\rho_E(g))\otimes (\upsilon_{-1}\upsilon_1)
    =\rho_A(\inn{f}{g})\otimes\upsilon_0
    =\hat\rho_A(\inn{f}{g}).
  \end{gather*}
  The pair $\hat\rho$ preverves left $A$-scalar multiplication since the following holds for any $a\in A$ and $f\in E$:
  \begin{gather*}
    \hat\rho_A(a)\hat\rho_E(f)
    =(\rho_A(a)\otimes\upsilon_0)(\rho_E(f)\otimes\upsilon_1)\\
    =(\rho_A(a)\rho_E(f))\otimes(\upsilon_0\upsilon_1)
    =\rho_E(af)\otimes\upsilon_1
    =\hat\rho_E(af).
  \end{gather*}
  Therefore, $\hat\rho$ is a representation of $(A,E)$.

  The injectivity of $\hat\rho$ follows immediately from the injectivity of $\rho$.

  We show that $\hat\rho$ is covariant.
  For any $f,g\in E$, we have 
  \begin{gather*}
    \psi_{\hat\rho}(\theta_{f,g})
    =\hat\rho_E(f)\hat\rho_E(g)^*
    =(\rho_E(f)\otimes\upsilon_1)(\rho_E(g)\otimes\upsilon_1)^*\\
    =(\rho_E(f)\otimes\upsilon_1)(\rho_E(g)^*\otimes\upsilon_{-1})
    =(\rho_E(f)\rho_E(g)^*)\otimes (\upsilon_1\upsilon_{-1})
    =\psi_\rho(\theta_{f,g})\otimes\upsilon_0.
  \end{gather*}
  So the equation $\psi_{\hat\rho}(\bullet)=\psi_\rho(\bullet)\otimes\upsilon_0$ holds on $\Theta(E)$.
  By the linearity and continuity, it also holds on $\K(E)$.
  For any $a$ in the covariant ideal $\I_E$ of $(A,E)$, we have
  \begin{equation*}
    \hat\rho_A(a)
    =\rho_A(a)\otimes\upsilon_0
    =\psi_\rho(\varphi(a))\otimes\upsilon_0
    =\psi_{\hat\rho}(\varphi(a)).
  \end{equation*}
  Therefore, $\hat\rho$ is covariant.

  We show that $\hat\rho$ admits a gauge action.
  Take any complex number $c$ with absolute value $1$.
  By Fact \ref{2.11}, there exists a $*$-homomorphism $\Phi_c\colon C^*(\Z)\to C^*(\Z)$ satisfying $\Phi_c(\upsilon_0)=\upsilon_0$ and $\Phi_c(\upsilon_1)=\upsilon_1 c$.
  The $*$-homomorphism $\id\otimes\Phi_c\colon B\otimes C^*(\Z)\to B\otimes C^*(\Z)$ satisfies that for any $a\in A$, 
  \begin{gather*}
    (\id\otimes\Phi_c)(\hat\rho_A(a))
    =(\id\otimes\Phi_c)(\rho_A(a)\otimes\upsilon_0)\\
    =\id(\rho_A(a))\otimes\Phi_c(\upsilon_0)
    =\rho_A(a)\otimes\upsilon_0
    =\hat\rho_A(a)
  \end{gather*}
  and that for any $f\in E$,
  \begin{gather*}
    (\id\otimes\Phi_c)(\hat\rho_E(f))
    =(\id\otimes\Phi_c)(\rho_E(f)\otimes\upsilon_1)\\
    =\id(\rho_E(f))\otimes\Phi_c(\upsilon_1)
    =\rho_E(f)\otimes\upsilon_1 c
    =c\hat\rho_E(f).
  \end{gather*}
  Therefore, $\hat\rho$ admits a gauge action.

  By the above discussion and Fact \ref{2.10}, $\hat\rho$ is a universal covariant representation.
\end{proof}

\begin{proposition}\label{2.4.3}
  Let $A$ be a C*-algebra.
  For all $n\in\Z$, the operator ${-}\otimes\upsilon_n$ from $A$ to $A\otimes C^*(\Z)$, defined by $a\mapsto a\otimes\upsilon_n$, is isometric.
\end{proposition}

\begin{proof}
  The operator ${-}\otimes\upsilon_0$ is an injective $*$-homomorphism and hence is isometric.
  For general $n\in\Z$, the operator ${-}\otimes\upsilon_n$ is also isometric, since the following holds for any $a\in A$: 
  \begin{equation*}
    \norm{a\otimes\upsilon_n}^2
    = \norm{(a\otimes\upsilon_n)^*(a\otimes\upsilon_n)}
    = \norm{a^*a\otimes\upsilon_0}
    = \norm{a^*a}
    = \norm{a}^2.
    \qedhere
  \end{equation*}
\end{proof}

\begin{fact}[\cite{I}, Proposition~2.16]\label{2.4.4}
  Let $\delta_\Z\colon M_\Z\to D_\Z$ be the faithful conditional expectation given by Fact~\ref{2.5}.
  For each $n\in\Z$, define a bounded linear functional $\delta_n\colon C^*(\Z)\to\C1_{M_\Z}\cong\C$ by $\delta_n(\alpha):=\delta_\Z(\alpha\upsilon_{-n})$.
  Let $B$ be a C*-algebra and $\xi$ be an element of the tensor product $B\otimes C^*(\Z)$.
  Assume that $(\id\otimes\delta_n)(\xi)=0$ for all $n\in\Z$.
  Then, $\xi = 0$.
\end{fact}

\section{Preparations for Iterated Function Systems}

\subsection{Definitions and Fundamental Properties}

The idea of regarding a fractal as a finite collection of contraction maps was introduced by Hutchinson in \cite{H}.
He showed that for finitely many contraction maps $\gamma_1$, \ldots, $\gamma_N$ on a complete metric space into itself, there exists a unique closed bounded subset $K$ such that $K=\gamma_1(K)\cup\cdots\cup\gamma_N(K)$; and moreover $K$ is compact.
Based on this result of Hutchinson, an iterated function system is often defined as a finite family of contraction maps on a complete metric space.
Furthermore, Kajiwara and Watatani impose the property called ``proper'', which is the property that there exists positive constants $c_1$, \ldots, $c_N$, $c'_1$, \ldots, $c'_N$ with $0<c_k\leq c'_k<1$ satisfying the condition
\begin{equation*}
  {c_k}d(x,y)
  \leq d(\gamma_k(x),\gamma_k(y))
  \leq {c'_k}d(x,y)
\end{equation*}
for any points $x$ and $y$ and index $k\in\{1, \ldots, N\}$.
However, in fact, the definition of Kajiwara–Watatani algebras does not require metric information; only topological information is sufficient.
Thus, in this paper, we define the iterated function system as follows.

\begin{definition}\label{3.1}
  Let $X$ be a topological space.
  An \emph{iterated function system} on $X$ is a finite family of continuous functions from $X$ into itself.
\end{definition}

\begin{remark}
  In this paper, iterated function systems allow repetitions.
\end{remark}

\begin{notation}
  Let $X$ be a topological space and $\gamma=\{\gamma_1, \ldots, \gamma_N\}$ be an iterated function system on $X$.
  For a finite sequence $\boldsymbol k = (k_n, \ldots, k_1)$ of the indices $1$, \ldots, $N$, let $\gamma_{\boldsymbol k}$, or $\gamma_{k_n, \ldots, k_1}$, denote the composition $\gamma_{k_n}\circ\cdots\circ\gamma_{k_1}$.
  We note that if $\boldsymbol k$ is the empty sequence, $\gamma_{\boldsymbol k}$ is the identity map on $X$.
  For a set $S$, let $\gamma(S)$ denote the union of the images of $S$ under the $\gamma_k$, that is, $\gamma(S):=\gamma_1(S)\cup\cdots\cup\gamma_N(S)$.
  $\gamma^n(S)$ is defined recursively by $\gamma^0(S):=S$ and $\gamma^{1+n}(S):=\gamma(\gamma^n(S))$.
  We note that 
  \begin{equation*}
    \gamma^n(S) = \set{\gamma_{\boldsymbol k}(x)}{\text{$x\in S$ and $\boldsymbol k$ is an index sequence of length $n$}}.
  \end{equation*}
\end{notation}

\begin{definition}\label{3.2}
  Let $X$ be a topological space and $\gamma=\{\gamma_1, \ldots, \gamma_N\}$ be an iterated function system on $X$.
  \emph{The graph of $\gamma$} is the subspace $G$ of the product space $X\times X$ defined by
  \begin{equation*}
    G:=\bigcup_{k=1}^{N}\set{(x,y)\in X\times X}{x=\gamma_k(y)}.
  \end{equation*}
\end{definition}

\begin{remark}
  If $X$ is a Hausdorff space, then the graph of $\gamma$ is closed in $X\times X$.
\end{remark}

\begin{definition}\label{3.3}
  Let $X$ be a topological space and $\gamma=\{\gamma_1, \ldots, \gamma_N\}$ be an iterated function system on $X$.
  \emph{The open set condition} for $\gamma$ is the condition that there exists a dense open subset $U$ of $X$ such that the images $\gamma_1(U)$, \ldots, $\gamma_N(U)$ are pairwise disjoint.
\end{definition}

\begin{remark}
  Typically, the open set condition also requires the inclusion 
  \begin{equation*}
    \gamma(U)
    =\gamma_1(U)\cup\cdots\cup\gamma_N(U)
    \subseteq U,
  \end{equation*}
  but in this paper, we do not impose this condition.
\end{remark}

\begin{definition}
  Let $X$ be a topological space and $\gamma=\{\gamma_1, \ldots, \gamma_N\}$ be an iterated function system on $X$.
  \emph{The graph separation condition} for $\gamma$ is the condition that for all distinct indices $j$ and $k$ and all points $x\in X$, one has $\gamma_j(x)\neq\gamma_k(x)$.
\end{definition}

\begin{definition}\label{3.4}
  Let $X$ be a topological space and $\gamma=\{\gamma_1, \ldots, \gamma_N\}$ be an iterated function system on $X$.
  The system $\gamma$ is called \emph{essentially free} if the set $\set{x\in X}{\gamma_{\boldsymbol k}(x)=\gamma_{\boldsymbol l}(x)}$ is nowhere dense for all finite sequence $\boldsymbol k$ and $\boldsymbol l$ of indices with $\boldsymbol k\neq\boldsymbol l$.
\end{definition}

\begin{definition}\label{3.5}
  Let $X$ be a topological space and $\gamma=\{\gamma_1, \ldots, \gamma_N\}$ be an iterated function system on $X$.
  A subset $S$ of $X$ is said to be \emph{completely invariant} under $\gamma$, or \emph{completely $\gamma$-invariant}, if both the image $\gamma_k(S)$ and the preimage $\gamma_k^{-1}(S)$ are contained in $S$ for all $k\in\{1, \ldots, N\}$.
\end{definition}

\begin{proposition}\label{3.6}
  Let $X$ be a topological space and $\gamma=\{\gamma_1, \ldots, \gamma_N\}$ be an iterated function system on $X$.
  Then, the family of all completely $\gamma$-invariant subsets of $X$ is closed under intersections.
\end{proposition}

\begin{proof}
  Let $\{S_\lambda\}_{\lambda\in\Lambda}$ be a family of completely $\gamma$-invariant subsets of $X$ and $S$ be its intersection.
  Then, for any $i\in\{1, \ldots, N\}$, we have
  \begin{equation*}
    \gamma_i(S)
    \subseteq\bigcap_{\lambda\in\Lambda}\gamma_i(S_\lambda)
    \subseteq\bigcap_{\lambda\in\Lambda}S_\lambda
    =S
  \end{equation*}
  and
  \begin{equation*}
    \gamma_i^{-1}(S)
    =\bigcap_{\lambda\in\Lambda}\gamma_i^{-1}(S_\lambda)
    \subseteq\bigcap_{\lambda\in\Lambda}S_\lambda
    =S.
  \end{equation*}
  Therefore, $S$ is also completely invariant under $\gamma$.
\end{proof}

\begin{definition}\label{3.7}
  Let $X$ be a topological space, $\gamma$ be an iterated function system on $X$, and $S$ be a subset of $X$.
  \emph{The completely $\gamma$-invariant subset generated by $S$} is the smallest one of completely $\gamma$-invariant subsets of $X$ which contains $S$.
\end{definition}

\begin{lemma}\label{3.8}
  Let $X$ be a topological space, $\gamma$ be a closed embedding of $X$ into itself, and $S$ be a nowhere dense subset of $X$.
  Then, the image $\gamma(S)$ is also nowhere dense in $X$.
\end{lemma}

\begin{proof}
  Let $K$ be the closure of $S$, which has empty interior.
  Since $\gamma$ is a closed map, the image $\gamma(K)$ is closed.
  Let $O$ be the interior of $\gamma(K)$.
  By the continuity of $\gamma$, the preimage $\gamma^{-1}(O)$ is open.
  By the injectivity of $\gamma$, $K$ coincides with $\gamma^{-1}(\gamma(K))$, which contains $\gamma^{-1}(O)$.
  So $\gamma^{-1}(O)$ is contained in the interior of $K$; and hence is empty.
  Since $O$ is contained in the range of $\gamma$, $O$ is also empty.
  Therefore, $\gamma(K)$ is a nowhere dense closed subset of $X$.
  Since $\gamma(S)$ is contained in $\gamma(K)$, it is nowhere dense.
\end{proof}

\begin{lemma}\label{3.9}
  Let $X$ be a topological space, $\gamma$ be a continuous map from $X$ into itself, $U$ be a dense open subset of $X$, and $S$ be a nowhere dense subset of $X$.
  Assume that the restriction of $\gamma$ to $U$ is an open map from $U$ into $X$.
  Then, the preimage $\gamma^{-1}(S)$ is also nowhere dense in $X$.
\end{lemma}

\begin{proof}
  Let $K$ be the closure of $S$, which has empty interior.
  By the continuity of $\gamma$, the preimage $\gamma^{-1}(K)$ is closed.
  Let $O$ be the interior of $\gamma^{-1}(K)$.
  The set $\gamma(U\cap O)$ is open in $X$.
  We have $\gamma(U\cap O)\subseteq\gamma(O)\subseteq\gamma(\gamma^{-1}(K))\subseteq K$.
  So $\gamma(U\cap O)$ is contained in the interior of $K$; and hence is empty.
  This implies that $U\cap O$ is also empty.
  Since $U$ is dense in $X$, the open subset $O$ of $X$ is empty.
  Therefore, $\gamma^{-1}(K)$ is a nowhere dense closed subset of $X$.
  Since $\gamma^{-1}(S)$ is contained in $\gamma^{-1}(K)$, it is nowhere dense.
\end{proof}

\begin{lemma}\label{3.10}
  Let $X$ be a topological space, $\gamma$ be a closed embedding of $X$ into itself, and $S$ be a meager subset of $X$.
  Assume that there exists a dense open subset $U$ of $X$ such that $\gamma(U)$ is open in $X$.
  Then, both the image $\gamma(S)$ and the preimage $\gamma^{-1}(S)$ are meager in $X$.
\end{lemma}

\begin{proof}
  The meager set $S$ can be written in the form of the union of countable many nowhere dense subsets $S_1$, $S_2$, \ldots of $X$.
  By Lemma \ref{3.8}, the images $\gamma(S_1)$, $\gamma(S_2)$, \ldots are nowhere dense and hence $\gamma(S)=\bigcup_{i}\gamma(S_i)$ is meager.
  Since $\gamma$ is a homeomorphism as a map from $X$ onto $\gamma(X)$ and the image $\gamma(U)$ is open in $X$, the restriction of $\gamma$ to $U$ is open as a map from $U$ into $X$.
  Thus, we can apply Lemma \ref{3.9}, which implies that the preimages $\gamma^{-1}(S_1)$, $\gamma^{-1}(S_2)$, \ldots are nowhere dense and hence $\gamma^{-1}(S)=\bigcup_{i}\gamma^{-1}(S_i)$ is meager.
\end{proof}

\begin{proposition}\label{3.11}
  Let $X$ be a topological space, $\gamma=\{\gamma_1, \ldots, \gamma_N\}$ be an iterated function system on $X$ into itself, and $S$ be a meager subset of $X$.
  Assume that $\gamma_1$, \ldots, $\gamma_N$ are closed embeddings and that there exists a dense open subset $U$ of $X$ such that all of $\gamma_1(U)$, \ldots, $\gamma_N(U)$ are open in $X$.
  Then, the completely $\gamma$-invariant subset of $X$ generated by $S$ is meager in $X$.
\end{proposition}

\begin{proof}
  Let $\mathcal P(X)$ denote the power set of $X$.
  Define a map $\Phi$ from $\mathcal P(X)$ to itself by
  \begin{equation*}
    \Phi(Y)=\gamma_1(Y)\cup\cdots\cup\gamma_N(Y)\cup\gamma_1^{-1}(Y)\cup\cdots\cup\gamma_N^{-1}(Y)
  \end{equation*}
  for $Y\in\mathcal P(X)$.
  Let $\Phi^n$ denote the $n$th iteration of $\Phi$.
  Let $S':=\bigcup_{n=0}^{\infty}\Phi^n(S)$.

  The set $S'$ contains $S=\Phi^0(S)$.
  For each $k\in\{1, \ldots, N\}$, we have
  \begin{equation*}
    \gamma_k(S')\cup\gamma_k^{-1}(S')
    \subseteq\Phi(S')
    =\bigcup_{n=0}^{\infty}\Phi(\Phi^n(S))
    =\bigcup_{n=1}^{\infty}\Phi^n(S)
    \subseteq S';
  \end{equation*}
  hence, $S'$ is completely invariant under $\gamma$.
  Therefore, $S'$ contains the completely $\gamma$-invariant subset of $X$ generated by $S$.

  By Lemma \ref{3.10}, the images $\gamma_1(S)$, \ldots, $\gamma_N(S)$ and the preimages $\gamma_1^{-1}(S)$, \ldots, $\gamma_N^{-1}(S)$ are meager; and hence, so is their union $\Phi(S)$.
  Similarly, $\Phi^2(S)$, $\Phi^3(S)$, \ldots are also meager.
  Therefore, $S'=\bigcup_{n=0}^{\infty}\Phi^n(S)$ is meager.
  Since the completely $\gamma$-invariant subset of $X$ generated by $S$ is contained in $S'$, it is also meager.
\end{proof}

\subsection{The binary operator $\boxdot$}

The goal of this subsection is to prove Proposition~\ref{3.15}.

\begin{definition}\label{3.13}
  Let $X$ be a topological space and $\gamma=\{\gamma_1, \ldots, \gamma_N\}$ be an iterated function system on $X$.
  \emph{The binary relation induced by $\gamma$} is the binary relation $\prec$ on $X$ defined by
  \begin{equation*}
    x\prec y
    \defiff x=\gamma_k(y)\text{ for some }k\in\{1, \ldots, N\}.
  \end{equation*}
\end{definition}

\begin{definition}\label{3.14}
  Let $X$ be a topological space and $\gamma$ be an iterated function system on $X$.
  Let $n\in\N_0$.
  A \emph{$\gamma$-chain of length $n$} is a non-empty finite sequence $(x_0, x_1, \ldots, x_n)$ in $X$ satisfying the condition
  \begin{equation*}
    x_0\prec x_1\prec\cdots\prec x_n.
  \end{equation*}
  Here, $\prec$ is the binary relation induced by $\gamma$.
\end{definition}

\begin{remark}
  If $X$ is a Hausdorff space, then the set of all the $\gamma$-chains of length $n$ is closed in the product space $X^{n+1}$.
\end{remark}

\begin{remark}
  The space of all the $\gamma$-chains of length $0$ coincides with $X$ and 
  the space of all the $\gamma$-chains of length $1$ coincides with the graph of $\gamma$.
\end{remark}

\begin{proposition}\label{3.15}
  Let $X$ be a compact Hausdorff space, and 
  let $\gamma=\{\gamma_1,\dots,\gamma_N\}$ be an iterated function system on $X$.
  Assume that all of $\gamma_1$, \ldots, $\gamma_N$ are embeddings.
  For each $n\in\N_0$, let $X_n$ be the space of $\gamma$-chains of length $n$.
  For functions $f$ on $X_m$ and $g$ on $X_n$, define a function $f\boxdot g$ on $X_{m+n}$ by
  \begin{equation*}
    (f\boxdot g)(x_{m+n}, \ldots, x_0):=f(x_{m+n}, \ldots, x_n)g(x_n, \ldots, x_0)
  \end{equation*}
  for $(x_{m+n}, \ldots, x_0)\in X_{m+n}$.
  Then, the norm-closed additive subgroup of $C(X_{m+n})$ generated by 
  \begin{equation*}
    C(X_m)\boxdot C(X_n)
    =\set{f\boxdot g}{f\in C(X_m),\ g\in C(X_n)}
  \end{equation*}
  coincides with the whole of $C(X_{m+n})$.
\end{proposition}

\begin{remark}
  By this proposition, the following equation holds for all $m,n\in\N_0$:
  \begin{equation*}
    \cspan\bigl(C(X_m)\boxdot C(X_n)\bigr)=C(X_{m+n}).
  \end{equation*}
\end{remark}

Consider the setting of Proposition~\ref{3.15}.
Fix $m$ and $n$.
Define two projection maps
\begin{equation*}
  \mathrm{pr}^L\colon X_{m+n}\to X_m
  \quad\text{and}\quad
  \mathrm{pr}^R\colon X_{m+n}\to X_n
\end{equation*}
by the following: for each point $\boldsymbol x=(x_{m+n}, \ldots, x_0)$ in $X_{m+n}$,
\begin{equation*}
  \mathrm{pr}^L(\boldsymbol x)=(x_{m+n}, \ldots, x_n)\quad\text{and}\quad
  \mathrm{pr}^R(\boldsymbol x)=(x_n, \ldots, x_0).
\end{equation*}  
For each tuple $\boldsymbol k=(k_m,\ldots,k_1)$ in $\{1,\ldots,N\}^m$, define a subset $Y_{\boldsymbol k}\subseteq X_{m+n}$ by
\begin{equation*}
  Y_{\boldsymbol k}
  :=\set{(x_{m+n}, \ldots, x_0)\in X_{m+n}}{x_{i+n}=\gamma_{k_i}(x_{i-1+n})\text{ for all }i=1,\ldots,m}.
\end{equation*}
By the continuity of $\gamma_1$, \ldots, $\gamma_N$, 
the set $Y_{\boldsymbol k}$ is closed in $X_{m+n}$ and hence compact.

Let $\mathrm{ex}_{\boldsymbol k}\colon X_n\to Y_{\boldsymbol k}$ be the continuous map defined by the following: for each $(x_n,\ldots,x_0)\in X_n$,
\begin{align*}
  \mathrm{ex}_{\boldsymbol k}(x_n,\ldots,x_0)
  =(\gamma_{k_m}\circ\cdots\circ\gamma_{k_1}(x_n),\ldots,\gamma_{k_1}(x_n),x_n,\ldots,x_0).
\end{align*}
We have $\mathrm{pr}^R\circ\mathrm{ex}_{\boldsymbol k}=\id_{X_n}$ and hence $\mathrm{pr}^R(Y_{\boldsymbol k})=X_n$.
Let $\mathrm{pr}^R_{\boldsymbol k}\colon Y_{\boldsymbol k}\to X_n$ be the restrictions of $\mathrm{pr}^R$ to $Y_{\boldsymbol k}$.
Then, $\mathrm{ex}_{\boldsymbol k}\circ\mathrm{pr}^R_{\boldsymbol k}=\id_{Y_{\boldsymbol k}}$ holds.
So the maps $\mathrm{pr}^R_{\boldsymbol k}$ and $\mathrm{ex}_{\boldsymbol k}$ are inverses of each other and hence homeomorphisms.

Let $Y^L_{\boldsymbol k}:=\mathrm{pr}^L(Y_{\boldsymbol k})\subseteq X_m$.
By the continuity of $\mathrm{pr}^L$, 
the set $Y^L_{\boldsymbol k}$ is compact and hence closed in $X_m$.

\begin{lemma}\label{3.16}
  The following inclusion holds for all $\boldsymbol k=(k_m,\ldots,k_1)$ in $\{1,\ldots,N\}^m$:
  \begin{equation*}
    Y_{\boldsymbol k}^L
    \subseteq\set{(x_m,\ldots,x_0)\in X_m}{x_i=\gamma_{k_i}(x_{i-1})\text{ for all }i=1,\ldots,m}.
  \end{equation*}
\end{lemma}

\begin{proof}
  Take any $\boldsymbol x=(x_m,\ldots,x_0)$ in $Y_{\boldsymbol k}^L$.
  By the definition of $Y_{\boldsymbol k}^L$, we can take a point $\boldsymbol y=(y_{m+n},\ldots,y_0)$ in $Y_{\boldsymbol k}$ satisfying $\boldsymbol x=\mathrm{pr}^L(\boldsymbol y)$.
  By the definition of $Y_{\boldsymbol k}$, we have $y_{i+n}=\gamma_{k_i}(y_{i-1+n})$ for all $i=1,\ldots,m$.
  Since $\boldsymbol x=\mathrm{pr}^L(\boldsymbol y)$, we have $x_i=\gamma_{k_i}(x_{i-1})$ for all $i=1,\ldots,m$.
\end{proof}

\begin{lemma}\label{3.17}
  Let $S\subseteq X_{m+n}$ and $\boldsymbol k\in\{1,\ldots,N\}^m$.
  If $\mathrm{pr}^L(S)$ intersects $Y_{\boldsymbol k}^L$, then $S$ intersects $Y_{\boldsymbol k}$.
\end{lemma}

\begin{proof}
  The tuple $\boldsymbol k$ can be written as $\boldsymbol k=(k_m,\ldots,k_1)$.
  Assume that $\mathrm{pr}^L(S)$ intersects $Y_{\boldsymbol k}^L$.
  Then, we can take a point $\boldsymbol x=(x_{m+n},\ldots,x_n)$ in the intersection of $\mathrm{pr}^L(S)$ and $Y_{\boldsymbol k}^L$.
  Since $\boldsymbol x$ is in $\mathrm{pr}^L(S)$, we can take a point $\boldsymbol y=(y_{m+n},\ldots,y_0)$ in $S$ satisfying $\boldsymbol x=\mathrm{pr}^L(\boldsymbol y)$.
  By $\boldsymbol x\in Y_{\boldsymbol k}^L$ and Lemma \ref{3.16}, we have $x_{i+n}=\gamma_{k_i}(x_{i-1+n})$ and hence $y_{i+n}=\gamma_{k_i}(y_{i-1+n})$ for all $i=1,\ldots,m$.
  Therefore, $\boldsymbol y$ is in $Y_{\boldsymbol k}$.
  Since $\boldsymbol y$ is also in $S$, $S$ intersects $Y_{\boldsymbol k}$.
\end{proof}

Fix a bijection $k\colon\{1,\ldots,N^m\}\to\{1,\ldots,N\}^m$.
We define a subset $Z_t\subseteq X_{m+n}$ for $t=0, 1, \ldots, N^m$ recursively by 
\begin{equation*}
  Z_t:=\begin{cases}
    \emptyset&\text{if}\ t=0,\\
    Z_{t-1}\cup Y_{k(t)}&\text{if}\ t\in\{1, \ldots, N^m\}.
  \end{cases}
\end{equation*}
The sets $Z_t$ are closed in $X_{m+n}$.
The following holds:
\begin{equation*}
  \emptyset
  =Z_0
  \subseteq Z_1
  \subseteq\cdots
  \subseteq Z_{N^m}
  =X_{m+n}.
\end{equation*}

\begin{lemma}\label{3.18}
  Let $S\subseteq X_{m+n}$ and $t\in\{1,\ldots,N^m\}$.
  If $\mathrm{pr}^L(S)$ intersects $\mathrm{pr}^L(Z_t)$, then $S$ intersects $Z_t$.
\end{lemma}

\begin{proof}
  Assume that $\mathrm{pr}^L(S)$ intersects $\mathrm{pr}^L(Z_t)$.
  Since $Z_t=Y_{k(1)}\cup\cdots\cup Y_{k(t)}$, we have
  \begin{gather*}
    \mathrm{pr}^L(Z_t)
    =\mathrm{pr}^L(Y_{k(1)}\cup\cdots\cup Y_{k(t)})\\
    =\mathrm{pr}^L(Y_{k(1)})\cup\cdots\cup\mathrm{pr}^L(Y_{k(t)})
    =Y_{k(1)}^L\cup\cdots\cup Y_{k(t)}^L.
  \end{gather*}
  So $\mathrm{pr}^L(S)$ intersects $Y_{k(s)}^L$ for some $s$.
  By Lemma \ref{3.17}, $S$ intersects $Y_{k(s)}$.
  Since $Y_{k(s)}\subseteq Z_t$, $S$ intersects $Z_t$.
\end{proof}

\begin{lemma}\label{3.19}
  Let $S\subseteq X_{m+n}$ and $t\in\{0,1,\ldots,N^m\}$.
  If $S$ is disjoint from $Z_t$, then $\mathrm{pr}^L(S)$ is disjoint from $\mathrm{pr}^L(Z_t)$.
\end{lemma}

\begin{proof}
  In the case of $t=0$, $\mathrm{pr}^L(S)$ is disjoint from $\mathrm{pr}^L(Z_t)$ since $Z_t=\emptyset$.
  In the other case, the conclusion coincides with the contrapositive of Lemma \ref{3.18}.
\end{proof}

\begin{proof}[Proof of Proposition~\ref{3.15}]
  In this proof, let $\casg{C(X_m)\boxdot C(X_n)}$ denote the norm-closed additive subgroup of $C(X_{m+n})$ generated by $C(X_m)\boxdot C(X_n)$.
  Obviously, the inclusion 
  \begin{equation*}
    \casg{C(X_m)\boxdot C(X_n)}\subseteq C(X_{m+n})
  \end{equation*}
  holds.
  
  It suffices to show the inverse inclusion ``$\supseteq$''.
  To this end, we show the following claim by backward induction on $t$, starting from $t=N^m$ and working downward to $t=0$: for all $t\in\{0, 1, \ldots, N^m\}$, one has the inclusion
  \begin{equation*}
    \set{f\in C(X_{m+n})}{\text{$f$ vanishes on $Z_t$}}
    \subseteq\casg{C(X_m)\boxdot C(X_n)}.
  \end{equation*}
  Since $Z_0=\emptyset$, the left-hand side becomes maximal, and this claim implies
  \begin{equation*}
    C(X_{m+n})
    \subseteq\cspan\casg{C(X_m)\boxdot C(X_n)}.
  \end{equation*}
  The desired inclusion ``$\supseteq$’’ follows.
  \smallskip

  \textbf{The base case.}
  In the case of $t=N^m$, we have $Z_t=X_{m+n}$ and hence
  \begin{equation*}
    \set{f\in C(X_{m+n})}{\text{$f$ vanishes on $Z_t$}}
    =\{0\}
    \subseteq\casg{C(X_m)\boxdot C(X_n)}.
  \end{equation*}

  \textbf{The induction step.}
  Take any $t\in\{1, \ldots, N^m\}$ and assume that
  \begin{equation*}
    \set{f\in C(X_{m+n})}{\text{$f$ vanishes on $Z_t$}}
    \subseteq\casg{C(X_m)\boxdot C(X_n)}.
  \end{equation*}
  Take any $f\in C(X_{m+n})$ satisfying $\supp f\cap Z_{t-1}=\emptyset$.
  By Lemma \ref{3.19}, two compact sets $\mathrm{pr}^L(\supp f)$ and $\mathrm{pr}^L(Z_{t-1})$ are disjoint.
  So we can take a continuous function $g$ on $X_m$ which takes value $1$ on $\mathrm{pr}^L(\supp f)$ and value $0$ on $\mathrm{pr}^L(Z_{t-1})$.
  Let $h:=f\circ\mathrm{ex}_{k(t)}$, which is a continuous function on $X_n$.

  The function $g\boxdot h$ is continuous and coincides with $f$ on $Y_{k(t)}$ and vanishes on $Z_{t-1}$.
  Indeed, for any point $\boldsymbol x$ in $Y_{k(t)}$, we have
  \begin{gather*}
    (g\boxdot h)(\boldsymbol x)
    =g(\mathrm{pr}^L(\boldsymbol x))h(\mathrm{pr}^R(\boldsymbol x))
    =g(\mathrm{pr}^L(\boldsymbol x))f(\mathrm{ex}_{k(t)}(\mathrm{pr}^R_{k(t)}(\boldsymbol x)))\\
    =g(\mathrm{pr}^L(\boldsymbol x))f(\boldsymbol x)
    =\begin{cases}
      1\cdot f(\boldsymbol x)&\text{if }\boldsymbol x\in\supp f\\
      g(\mathrm{pr}^L(\boldsymbol x))\cdot 0&\text{otherwise}
    \end{cases}
    =f(\boldsymbol x);
  \end{gather*}
  and for any $\boldsymbol x\in Z_{t-1}$, we have
  \begin{gather*}
    (g\boxdot h)(\boldsymbol x)
    =g(\mathrm{pr}^L(\boldsymbol x))h(\mathrm{pr}^R(\boldsymbol x))
    =0\cdot h(\mathrm{pr}^R(\boldsymbol x))
    =0.
  \end{gather*}
  So $f-g\boxdot h$ is a continuous function on $X_{m+n}$ vanishing on $Z_{t-1}\cup Y_{k(t)}=Z_t$ and hence by the induction hypothesis, belongs to $\casg{C(X_m)\boxdot C(X_n)}$.
  Since $g\boxdot h$ belongs to $C(X_m)\boxdot C(X_n)$, $f$ also belongs to $\casg{C(X_m)\boxdot C(X_n)}$.

  By the above discussion in the induction step, we have the inclusion
  \begin{equation*}
    \set{f\in C(X_{m+n})}{\supp f\cap Z_{t-1}=\emptyset}
    \subseteq\casg{C(X_m)\boxdot C(X_n)}.
  \end{equation*}
  By taking the closure of both sides, we obtain the inclusion
  \begin{equation*}
    \set{f\in C(X_{m+n})}{\text{$f$ vanishes on $Z_{t-1}$}}
    \subseteq\casg{C(X_m)\boxdot C(X_n)}.
  \end{equation*}
  This is what we aimed to show in the induction step.
\end{proof}

\subsection{The binary operator $\boxtimes$}

The goal of this subsection is to prove Proposition~\ref{3.21}.

\begin{definition}\label{3.20}
  Let $X$ be a topological space and $\gamma=\{\gamma_1, \ldots, \gamma_N\}$ be an iterated function system on $X$.
  Let $m,n\in\N_0$.
  A $\gamma$-bichain of length $(m,n)$ is a pair $(\boldsymbol x, \boldsymbol y)$ of $\gamma$-chains $\boldsymbol x = (x_m, \ldots, x_0)$ of length $m$ and $\boldsymbol y = (y_n, \ldots, y_0)$ of length $n$ satisfying $x_0 = y_0$.
\end{definition}

\begin{remark}
  If $X$ is a Hausdorff space, then the set of all the $\gamma$-bichains of length $(m,n)$ is closed in the product space $X_m\times X_n$.
\end{remark}

\begin{proposition}\label{3.21}
  Let $X$ be a compact Hausdorff space and $\gamma=\{\gamma_1, \ldots, \gamma_N\}$ be an iterated function system on $X$.
  For each $n\in\N_0$, let $X_n$ be the space of all the $\gamma$-chains of length $n$.
  For each $m,n\in\N_0$, let $X_{m,n}$ be the space of all the $\gamma$-bichains of length $(m,n)$.
  For functions $f$ on $X_m$ and $g$ on $X_n$, define a function $f\boxtimes g$ on $X_{m,n}$ by
  \begin{equation*}
    (f\boxtimes g)(\boldsymbol z):=f(\boldsymbol x)\overline{g(\boldsymbol y)}
  \end{equation*}
  for $\boldsymbol z = (\boldsymbol x,\boldsymbol y)\in X_{m,n}$.
  Then, one has
  \begin{equation*}
    \cspan\bigl(C(X_m)\boxtimes C(X_n)\bigr)=C(X_{m,n}).
  \end{equation*}
  Here, the closed linear span on the left-hand side is with respect to the uniform norm; and $C(X_m)\boxtimes C(X_n)$ denotes the set $\set{f\boxtimes g}{f\in C(X_m), g\in C(X_n)}$.
\end{proposition}

Consider the setting of Proposition~\ref{3.21}.
Fix $m$ and $n$.
Define two projection maps
\begin{equation*}
  \mathrm{pr}^L\colon X_{m,n}\to X_m
  \quad\text{and}\quad
  \mathrm{pr}^R\colon X_{m,n}\to X_n
\end{equation*}
by the following: for each $\boldsymbol z = (\boldsymbol x, \boldsymbol y)\in X_{m,n}$,
\begin{equation*}
  \mathrm{pr}^L(\boldsymbol z)=\boldsymbol x\quad\text{and}\quad
  \mathrm{pr}^R(\boldsymbol z)=\boldsymbol y.
\end{equation*}
For each tuple $\boldsymbol k=(k_n,\ldots,k_1)$ in $\{1,\ldots,N\}^n$, define subsets $Y^R_{\boldsymbol k}\subseteq X_n$ and $Y_{\boldsymbol k}\subseteq X_{m,n}$ by
\begin{align*}
  Y^R_{\boldsymbol k}
  &:=\set{(y_n,\ldots,y_0)\in X_n}{y_i=\gamma_{k_i}(y_{i-1})\text{ for all }i=1,\ldots,n}\ \text{and}\\
  Y_{\boldsymbol k}
  &:=\set{\boldsymbol z\in X_{m,n}}{\mathrm{pr}^R(\boldsymbol z)\in Y^R_{\boldsymbol k}}.
\end{align*}
By continuity of $\gamma_1,\ldots,\gamma_N$,
the set $Y^R_{\boldsymbol k}$ is closed in $X_n$ and hence compact.
By continuity of $\mathrm{pr}^R$,
the set $Y_{\boldsymbol k}$ is closed in $X_{m,n}$ and hence compact.

Let $\mathrm{ex}_{\boldsymbol k}\colon X_m\to Y_{\boldsymbol k}$ be the continuous map defined by the following: for each $\boldsymbol x=(x_m,\ldots,x_0)$ in $X_m$,
\begin{equation*}
  \mathrm{ex}_{\boldsymbol k}(\boldsymbol x)
  :=(\boldsymbol x,(\gamma_{k_n}\circ\cdots\circ\gamma_{k_1}(x_0),\ldots,\gamma_{k_1}(x_0),x_0)).
\end{equation*}
We have $\mathrm{pr}^L\circ\mathrm{ex}_{\boldsymbol k}=\id_{X_m}$ and hence $\mathrm{pr}^L(Y_{\boldsymbol k})=X_m$.
Let $\mathrm{pr}^L_{\boldsymbol k}\colon Y_{\boldsymbol k}\to X_m$ be the restriction of $\mathrm{pr}^L$ to $Y_{\boldsymbol k}$.
Then, $\mathrm{ex}_{\boldsymbol k}\circ\mathrm{pr}^L_{\boldsymbol k}=\id_{Y_{\boldsymbol k}}$ holds.
So the maps $\mathrm{pr}^L_{\boldsymbol k}$ and $\mathrm{ex}_{\boldsymbol k}$ are inverses of each other and hence homeomorphisms.

\begin{lemma}\label{3.22}
  The projection $\mathrm{pr}^R\colon X_{m,n}\to X_n$ is surjective.
\end{lemma}

\begin{proof}
  For all $\boldsymbol y=(y_n,\ldots,y_0)$ in $X_n$,
  the point
  \begin{equation*}
    \boldsymbol x:=(\underbrace{\gamma_1\circ\cdots\circ\gamma_1}_m(y_0),\ldots,\gamma_1(y_0),y_0)
  \end{equation*}
  is in $X_m$ and the pair $\boldsymbol z:=(\boldsymbol x,\boldsymbol y)$ is in $X_{m,n}$.
  Therefore, $\mathrm{pr}^R$ is surjective.
\end{proof}

\begin{lemma}\label{3.23}
  The equation $\mathrm{pr}^R(Y_{\boldsymbol k})=Y^R_{\boldsymbol k}$ holds for all tuple $\boldsymbol k$ in $\{1,\ldots,N\}^n$.
\end{lemma}

\begin{proof}
  Since $Y_{\boldsymbol k}$ is the preimage of $Y^R_{\boldsymbol k}$ under $\mathrm{pr}^R$, the desired equation follows directly from the surjectivity of $\mathrm{pr}^R$ (Lemma~\ref{3.22}).
\end{proof}

\begin{lemma}\label{3.24}
  Let $S\subseteq X_{m,n}$ and $\boldsymbol k\in\{1,\ldots,N\}^n$.
  If $\mathrm{pr}^R(S)$ intersects $Y_{\boldsymbol k}^R$,
  then $S$ intersects $Y_{\boldsymbol k}$.
\end{lemma}

\begin{proof}
  Assume that $\mathrm{pr}^R(S)$ intersects $Y_{\boldsymbol k}^R$.
  Then, we can take a point $\boldsymbol y$ in the intersection of $\mathrm{pr}^R(S)$ and $Y_{\boldsymbol k}^R$.
  Since $\boldsymbol y\in\mathrm{pr}^R(S)$, we can take a point $\boldsymbol z$ in $S$ satisfying $\boldsymbol y=\mathrm{pr}^R(\boldsymbol z)$.
  Since $\mathrm{pr}^R(\boldsymbol z)\in Y^R_{\boldsymbol k}$, the point $\boldsymbol z$ is in $Y_{\boldsymbol k}$.
  Therefore, $S$ intersects $Y_{\boldsymbol k}$ at $\boldsymbol z$.
\end{proof}

Fix a bijection $k\colon\{1,\ldots,N^n\}\to\{1,\ldots,N\}^n$.
We define a subset $Z_t\subseteq X_{m,n}$ for $t=0,1,\ldots,N^n$ recursively by
\begin{equation*}
  Z_t:=\begin{cases}
    \emptyset&\text{if}\ t=0,\\
    Z_{t-1}\cup Y_{k(t)}&\text{if}\ t\in\{1,\ldots,N^n\}.
  \end{cases}
\end{equation*}
The sets $Z_t$ are closed in $X_{m,n}$.
The following holds:
\begin{equation*}
  \emptyset
  =Z_0
  \subseteq Z_1
  \subseteq\cdots
  \subseteq Z_{N^n}
  =X_{m,n}.
\end{equation*}

\begin{lemma}\label{3.25}
  Let $S\subseteq X_{m,n}$ and $t\in\{1,\ldots,N^n\}$.
  If $\mathrm{pr}^R(S)$ intersects $\mathrm{pr}^R(Z_t)$,
  then $S$ intersects $Z_t$.
\end{lemma}

\begin{proof}
  Assume that $\mathrm{pr}^R(S)$ intersects $\mathrm{pr}^R(Z_t)$.
  Since $Z_t=Y_{k(1)}\cup\cdots\cup Y_{k(t)}$, by Lemma~\ref{3.23}, we have
  \begin{gather*}
    \mathrm{pr}^R(Z_t)
    =\mathrm{pr}^R(Y_{k(1)}\cup\cdots\cup Y_{k(t)})\\
    =\mathrm{pr}^R(Y_{k(1)})\cup\cdots\cup\mathrm{pr}^R(Y_{k(t)})
    =Y^R_{k(1)}\cup\cdots\cup Y^R_{k(t)}.
  \end{gather*}
  So $\mathrm{pr}^R(S)$ intersects $Y^R_{k(s)}$ for some $s$.
  By Lemma~\ref{3.24}, $S$ intersects $Y_{k(s)}$.
  Since $Y_{k(s)}\subseteq Z_t$, $S$ intersects $Z_t$.
\end{proof}

\begin{lemma}\label{3.26}
  Let $S\subseteq X_{m,n}$ and $t\in\{0,1,\ldots,N^n\}$.
  If $S$ is disjoint from $Z_t$, then $\mathrm{pr}^R(S)$ is disjoint from $\mathrm{pr}^R(Z_t)$.
\end{lemma}

\begin{proof}
  In the case of $t=0$, $\mathrm{pr}^R(S)$ is disjoint from $\mathrm{pr}^R(Z_t)$ since $Z_t=\emptyset$.
  In the other case, the conclusion coincides with the contrapositive of Lemma~\ref{3.25}.
\end{proof}

\begin{proof}[Proof of Proposition~\ref{3.21}]
  We show the inclusion ``$\subseteq$’’, i.e.,
  \begin{equation*}
    \cspan\bigl(C(X_m)\boxtimes C(X_n)\bigr)\subseteq C(X_{m,n}).
  \end{equation*}
  For any functions $f$ on $X_m$ and $g$ on $X_n$, if $f$ and $g$ are continuous, then so is $f\boxtimes g$.
  Therefore, $C(X_m)\boxtimes C(X_n)$ is contained in $C(X_{m,n})$ and so is its closed linear span.

  We show the inverse inclusion ``$\supseteq$''.
  To this end, we show the following claim by backward induction on $t$, starting from $t=N^n$ and working downward to $t=0$: for all $t\in\{0,1,\ldots,N^n\}$, one has the inclusion
  \begin{equation*}
    \set{f\in C(X_{m,n})}{\text{$f$ vanishes on $Z_t$}}
    \subseteq\cspan\bigl(C(X_m)\boxtimes C(X_n)\bigr).
  \end{equation*}
  Since $Z_0=\emptyset$, the left-hand side becomes maximal, and this claim implies
  \begin{equation*}
    C(X_{m,n})
    \subseteq\cspan\bigl(C(X_m)\boxtimes C(X_n)\bigr).
  \end{equation*}
  The desired inclusion ``$\supseteq$'' follows.

  \textbf{The base case.}
  In the case of $t=N^n$, we have $Z_t=X_{m,n}$ and hence
  \begin{equation*}
    \set{f\in C(X_{m,n})}{\text{$f$ vanishes on $Z_t$}}
    =\{0\}
    \subseteq\cspan\bigl(C(X_m)\boxtimes C(X_n)\bigr).
  \end{equation*}

  \textbf{The induction step.}
  Take any $t\in\{1,\ldots,N^n\}$ and assume that
  \begin{equation*}
    \set{f\in C(X_{m,n})}{\text{$f$ vanishes on $Z_t$}}
    \subseteq\cspan\bigl(C(X_m)\boxtimes C(X_n)\bigr).
  \end{equation*}
  Take any $f\in C(X_{m,n})$ satisfying $\supp f\cap Z_{t-1}=\emptyset$.
  Let $g:=f\circ\mathrm{ex}_{k(t)}$, which is a continuous function on $X_m$.
  By Lemma~\ref{3.26}, two compact sets $\mathrm{pr}^R(\supp f)$ and $\mathrm{pr}^R(Z_{t-1})$ are disjoint.
  So we can take a continuous function $h$ on $X_m$ which takes value $1$ on $\mathrm{pr}^R(\supp f)$ and value $0$ on $\mathrm{pr}^R(Z_{t-1})$.

  The function $g\boxtimes h$ is continuous and coincides with $f$ on $Y_{k(t)}$ and vanishes on $Z_{t-1}$.
  Indeed, for any $\boldsymbol z\in Y_{k(t)}$, we have
  \begin{gather*}
    (g\boxtimes h)(\boldsymbol z)
    =g(\mathrm{pr}^L(\boldsymbol z))\overline{h(\mathrm{pr}^R(\boldsymbol z))}
    =f(\mathrm{ex}_{k(t)}(\mathrm{pr}^L_{k(t)}(\boldsymbol z)))\overline{h(\mathrm{pr}^R(\boldsymbol z))}\\
    =f(\boldsymbol z)\overline{h(\mathrm{pr}^R(\boldsymbol z))}
    =\begin{cases}
      f(\boldsymbol z)\cdot 1&\text{if }\boldsymbol z\in\supp f\\
      0\cdot\overline{h(\mathrm{pr}^R(\boldsymbol z))}&\text{otherwise}
    \end{cases}
    =f(\boldsymbol z);
  \end{gather*}
  and for any $\boldsymbol z\in Z_{t-1}$, we have
  \begin{gather*}
    (g\boxtimes h)(\boldsymbol z)
    =g(\mathrm{pr}^L(\boldsymbol z))\overline{h(\mathrm{pr}^R(\boldsymbol z))}
    =g(\mathrm{pr}^L(\boldsymbol z))\cdot\overline{0}
    =0.
  \end{gather*}
  So $f-g\boxtimes h$ is a continuous function on $X_{m,n}$ vanishing on $Z_{t-1}\cup Y_{k(t)}=Z_t$ and hence by the induction hypothesis, belongs to $\cspan\bigl(C(X_m)\boxtimes C(X_n)\bigr)$.
  Since $g\boxtimes h$ belongs to $C(X_m)\boxtimes C(X_n)$, $f$ also belongs to $\cspan\bigl(C(X_m)\boxtimes C(X_n)\bigr)$.

  By the above discussion in the induction step, we have the inclusion
  \begin{equation*}
    \set{f\in C(X_{m,n})}{\supp f\cap Z_{t-1}=\emptyset}
    \subseteq\cspan\bigl(C(X_m)\boxtimes C(X_n)\bigr).
  \end{equation*}
  By taking the closure of both sides, we obtain the inclusion
  \begin{equation*}
    \set{f\in C(X_{m,n})}{\text{$f$ vanishes on $Z_{t-1}$}}
    \subseteq\cspan\bigl(C(X_m)\boxtimes C(X_n)\bigr).
  \end{equation*}
  This is what we aimed to show in the induction step.
\end{proof}

\section{Review of Kajiwara–Watatani Algebras}

In this section, we briefly review the definition of C*-algebras associated with iterated function systems introduced in \cite{KW06}.
We present a slightly more generalized version than the original.

Let $X$ be a compact Hausdorff space, $\gamma=\{\gamma_1, \ldots, \gamma_N\}$ be an iterated function system on $X$, and $G$ be the graph of $\gamma$.
Let $A:=C(X)$ as a C*-algebra and let $E$ be the completion of the pre-Hilbert $A$-$A$-bimodule defined as the complex-linear space $C(G)$ endowed with the left and right $A$-scalar multiplication and $A$-valued inner product defined by the following.
\begin{equation*}
  \begin{array}{l}
    (a\cdot f)(x,y):=a(x)f(x,y)\\
    (f\cdot a)(x,y):=f(x,y)a(y)
  \end{array}
  \quad\text{for each $a\in A$, $f\in C_c(G)$ and $(x,y)\in G$}.
\end{equation*}
\begin{equation*}
  \inn{f}{g}(y):=\sum_{k=1}^{N}\overline{f(\gamma_k(y),y)}g(\gamma_k(y),y)\quad\text{for each $f,g\in C_c(G)$ and $y\in X$}.
\end{equation*}
Then, $E$ is a C*-correspondence over~$A$.
This is called \emph{the Kajiwara--Watatani correspondence for $(X,\gamma)$}; and the Cuntz--Pimsner--Katsura algebra associated with this C*-correspondence is called \emph{the Kajiwara--Watatani algebra for $\gamma$} and denoted by $\O_\gamma$.

\begin{remark}
  For any $f\in C(G)$, the inequality $\norm{f}_{C(G)}\leq\norm{f}_E\leq\sqrt{N}\cdot\norm{f}_{C(G)}$ holds and hence the two norms on $C(G)$ are equivalent.
  Therefore, $E$ coincides with $C(G)$ as a complex-linear space.
\end{remark}

\begin{proposition}\label{4.1}
  Let $\varphi\colon A\to\L(E)$ be the natural left action of $A$ on $E$.
  Then, the following equation holds:
  \begin{equation*}
    \ker\varphi=\set{b\in A}{b\text{ vanishes on }\gamma(X)}.
  \end{equation*}
\end{proposition}

\begin{proof}
  First, we show the inclusion ``$\subseteq$''.
  Take any $b\in\ker\varphi$.
  Take any $k\in\{1, \ldots, N\}$ and $x\in\gamma_k(X)$.
  We can take a point $y\in X$ satisfying $x=\gamma_k(y)$.
  Then, the point $(x,y)$ belongs to $G$.
  There exists a function $f\in E=C(G)$ taking value $1$ at the point $(x,y)$.
  We have
  \begin{equation*}
    b(x)
    =b(x)f(x,y)
    =(bf)(x,y)
    =0.
  \end{equation*}
  Therefore, $b$ vanishes on $\gamma(X)$.

  Next, we show the inverse inclusion ``$\supseteq$''.
  Take any $b\in A$ which vanishes on $\gamma(X)$.
  Take any $f\in E$ and $(x,y)\in G$.
  By the definition of $G$, $x=\gamma_k(y)$ holds for some $k$; and hence we have $x\in\gamma(X)$.
  So we have
  \begin{equation*}
    \varphi(b)(f)(x,y)
    =(bf)(x,y)
    =b(x)f(x,y)
    =0\cdot f(x,y)
    =0.
  \end{equation*}
  Therefore, $\varphi(b)=0$ and hence $b\in\ker\varphi$.
\end{proof}

\section{Certain Representations of Kajiwara--Watatani Correspondences}

In this section, we consider the following setting.
Let $X$ be a compact Hausdorff space, 
    $\gamma=\{\gamma_1,\ldots,\gamma_N\}$ be an iterated function system, and 
    $\X$ be a subset of $X$.
Let $G$ be the graph of $\gamma$ and 
    $(A,E)$ the Kajiwara--Watatani correspondence for $(X,\gamma)$.
We assume the following conditions (I) -- (III):
\begin{itemize}
  \item[(I)] The set $\X$ is dense in $X$ and completely invariant under $\gamma$.
  \item[(II)] The images $\gamma_1(\X),\ldots,\gamma_N(\X)$ are pairwise disjoint.
  \item[(III)] The restriction $\gamma_k|_\X$ is injective for all $k=1,\ldots,N$.
\end{itemize}
Under these conditions, the following statements hold:
\begin{itemize}
  \item The map $\gamma_{\boldsymbol k}$ is injective on $\X$ for all finite sequences $\boldsymbol k$ of indices.
  \item Let $\boldsymbol k = (k_n, \ldots, k_0)$ and $\boldsymbol k’ = (k'_{n'}, \ldots, k'_0)$ be finite sequences of indices.
  If $\boldsymbol k$ is an initial segment of $\boldsymbol k'$, i.e., $n \leq n'$ and $k_n = k'_{n'}$, \dots, $k_0 = k'_{n' - n}$ hold, one has $\gamma_{\boldsymbol k}(\X)\supseteq\gamma_{\boldsymbol k’}(\X)$.
  If neither $\boldsymbol k$ is an initial segment of $\boldsymbol k’$ nor $\boldsymbol k’$ is an initial segment of $\boldsymbol k$, $\gamma_{\boldsymbol k}(\X)$ and $\gamma_{\boldsymbol k’}(\X)$ are disjoint.
\end{itemize}
Let $\rho_A\colon A\to M_\X$ and $\rho_E\colon E\to M_\X$ be the maps defined below:
\begin{alignat*}{3}
  &\rho_A(a)\e_x
  &&:=\e_xa(x)
  &\quad&\text{for each $a\in A=C(X)$ and $x\in\X$; and}\\
  &\rho_E(f)\e_y
  &&:=\sum_{k=1}^{N}\e_{\gamma_k(y)}f(\gamma_k(y),y)
  &\quad&\text{for each $f\in E=C(G)$ and $y\in\X$.}
\end{alignat*}
We denote the element of the group C*-algebra $C^*(\Z)$ corresponding to $n\in\Z$ by $\upsilon_n$, as in subsection \ref{subsection 2.4}.
Let $\hat\rho_A\colon A\to M_\X\otimes C^*(\Z)$ and $\hat\rho_E\colon E\to M_\X\otimes C^*(\Z)$ be the maps defined below:
\begin{alignat*}{3}
  \hat\rho_A(a)&:=\rho_A(a)\otimes\upsilon_0&&\quad\text{for each $a\in A$; and}\\
  \hat\rho_E(f)&:=\rho_E(f)\otimes\upsilon_1&&\quad\text{for each $f\in E$.}
\end{alignat*}
In subsection \ref{subsection 5.2} and subsequent subsections, we assume the following additional condition:
\begin{itemize}
  \item All of $\gamma_1$, \ldots, $\gamma_N$ are embeddings.
\end{itemize}
In subsection \ref{subsection 5.5} and subsequent subsections, we assume the following additional conditions:
\begin{itemize}
  \item The iterated function system $\gamma$ satisfies the open set condition. That is, there exists a dense open set $U$ in $X$ such that the images $\gamma_1(U),\ldots,\gamma_N(U)$ are pairwise disjoint open set in $X$ and contained in $U$.
  \item The set $\X$ is contained in $U$.
  \item The set $\X$ is comeager in $X$.
\end{itemize}

\subsection{The pair $\rho$ is an injective covariant representation}

The goal of this subsection is to prove the following two theorems.

\begin{theorem}\label{5.1}
  The pair $\rho=(\rho_A,\rho_E)$ is an injective covariant representation of $(A,E)$.
\end{theorem}

\begin{theorem}\label{5.2}
  The pair $\hat\rho=(\hat\rho_A,\hat\rho_E)$ is universal as a covariant representation of $(A,E)$ and $C^*(\hat\rho)$ is isomorphic to $\O_\gamma$.
\end{theorem}

\begin{remark}
  By Theorem \ref{5.2}, we can regard the Kajiwara--Watatani algebra $\O_E$ as $C^*(\hat\rho)$, which is a closed $*$-subalgebra of $M_\X\otimes C^*(\Z)$.
\end{remark}

\begin{proposition}\label{5.3}
  The pair $\rho$ is a representation of the C*-correspondence $(A,E)$.
\end{proposition}

\begin{proof}
  It is easy to verify that $\rho_A$ is a $*$-homomorphism and that $\rho_E$ is a complex linear map.

  We show that $\rho$ preserves the inner product.
  Take any $f,g\in E$ and $y,y'\in\X$.
  We want to show the equation $\inn{\e_y}{\rho_E(f)^*\rho_E(g)\e_{y'}}=\inn{\e_y}{\rho_A(\inn{f}{g})\e_{y'}}$.
  We have
  \begin{align*}
    (\text{LHS})
    &=\inn{\rho_E(f)\e_y}{\rho_E(g)\e_{y'}}\\
    &=\INN\bigg{\sum_{j=1}^{N}\e_{\gamma_j(y)}f(\gamma_j(y),y)}{\sum_{k=1}^{N}\e_{\gamma_k(y')}g(\gamma_k(y'),y')}\\
    &=\sum_{j=1}^{N}\sum_{k=1}^{N}\overline{f(\gamma_j(y),y)}\inn{\e_{\gamma_j(y)}}{\e_{\gamma_k(y')}}g(\gamma_k(y'),y').\\
    \intertext{Since the sets $\gamma_1(\X)$, \ldots, $\gamma_N(\X)$ are pairwise disjoint, we have}
    &=\sum_{k=1}^{N}\overline{f(\gamma_k(y),y)}\inn{\e_{\gamma_k(y)}}{\e_{\gamma_k(y')}}g(\gamma_k(y'),y').\\
    \intertext{Since the restrictions $\gamma_k|_\X$ are injective, we have}
    &=\sum_{k=1}^{N}\overline{f(\gamma_k(y),y)}\inn{\e_y}{\e_{y'}}g(\gamma_k(y'),y')\\
    &=\inn{\e_y}{\e_{y'}}\sum_{k=1}^{N}\overline{f(\gamma_k(y'),y')}g(\gamma_k(y'),y')\\
    &=\inn{\e_y}{\e_{y'}}\inn{f}{g}(y')\\
    &=(\text{RHS}).
  \end{align*}
  So $(\text{LHS})=(\text{RHS})$ holds.
  Since it holds for any $y,y'\in\X$, we have $\rho_E(f)^*\rho_E(g)=\rho_A(\inn{f}{g})$.
  Therefore, $\rho$ preserves the inner product.

  We show that $\rho$ preserves left $A$-scalar multiplication.
  Take any $f,g\in E$.
  For any $y\in\X$, we have
  \begin{gather*}
    \rho_E(af)\e_y
    =\sum_{k=1}^{N}\e_{\gamma_k(y)}(af)(\gamma_k(y),y)
    =\sum_{k=1}^{N}\e_{\gamma_k(y)}a(\gamma_k(y))f(\gamma_k(y),y)\\
    =\sum_{k=1}^{N}\rho_A(a)\e_{\gamma_k(y)}f(\gamma_k(y),y)
    =\rho_A(a)\sum_{k=1}^{N}\e_{\gamma_k(y)}f(\gamma_k(y),y)
    =\rho_A(a)\rho_E(f)\e_y.
  \end{gather*}
  So we obtain that $\rho_E(af)=\rho_A(a)\rho_E(f)$.
  Therefore, $\rho$ preserves left $A$-scalar multiplication.

  By the above discussion, $\rho$ is a representation of $E$.
\end{proof}

\begin{proposition}\label{5.4}
  The representation $\rho$ is injective.
\end{proposition}

\begin{proof}
  We only show the injectivity of $\rho_A$.
  Take any $a\in\ker\rho_A$.
  For any point $x\in\X$, we have
  \begin{equation*}
    a(x)
    =\inn{\e_x}{\rho_A(a)\e_x}
    =\inn{\e_x}{0\e_x}
    =0.
  \end{equation*}
  Since $\X$ is dense in $X$, we have $a=0$.
  Thus, $\rho_A$ is injective.
\end{proof}

\begin{lemma}\label{5.5}
  Let $\varphi$ be the natural left action of $A$ on $E$ and $\alpha$ be an $\X$-squared matrix.
  Assume the following two conditions:
  \begin{itemize}
    \item $\alpha\rho_E(f)=0$ holds for all $f\in E$; and
    \item $\alpha\rho_A(b)=0$ holds for all $b\in\ker\varphi$.
  \end{itemize}
  Then, $\alpha=0$ holds.
\end{lemma}

\begin{proof}
  Take any $k\in\{1, \ldots, N\}$ and $x\in\X\cap\gamma_k(X)$.
  Let $y:=(\gamma_k|_\X)^{-1}(x)\in\X$.
  There exists a function $f\in C(G)$ with the following conditions:
  \begin{equation*}
    f(x,y)=f(\gamma_k(y),y)=1\text{ and }
    f(\gamma_j(y),y)=0\text{ for all }j\in\{1, \ldots, N\}\setminus\{k\}.
  \end{equation*}
  We have
  \begin{equation*}
    \alpha\e_x
    =\alpha\sum_{j=1}^{N}\e_{\gamma_j(y)}f(\gamma_j(y),y)
    =\alpha\rho_E(f)\e_y
    =0\e_y
    =\boldsymbol 0.
  \end{equation*}
  Therefore, the equation $\alpha\e_x=\boldsymbol 0$ holds for all $x\in\X\cap\gamma(X)$.

  Take any point $x\in\X\setminus\gamma(X)$.
  The set $\gamma(X)$ is compact, since it is the finite union of the images of the compact set $X$ under the continuous maps $\gamma_k$.
  So there exists a function $b\in A$ which takes value $1$ at $x$ and value $0$ on $\gamma(X)$.
  By Proposition~\ref{4.1}, we have $b\in\ker\varphi$.
  We obtain
  \begin{equation*}
    \alpha\e_x
    =\alpha\e_xb(x)
    =\alpha\rho_A(b)\e_x
    =0\e_x
    =\boldsymbol 0.
  \end{equation*}
  Therefore, the equation $\alpha\e_x=\boldsymbol 0$ holds for all $x\in\X\setminus\gamma(X)$.

  By the above discussion, $\alpha\e_x=\boldsymbol 0$ holds for all $x\in\X$ and hence $\alpha=0$.
\end{proof}

\begin{proposition}\label{5.6}
  The representation $\rho$ is covariant.
\end{proposition}

\begin{proof}
  In this proof, let 
    $\varphi$ be the natural left action of $A$ on $E$, 
    $\psi_\rho$ be the representation of $\K(E)$ induced by $\rho$, and 
    $\I_E$ be the covariant ideal of $E$.
  Take any $a\in\I_E$.
  By Fact~\ref{2.6}, for any $f\in E$ we have the equation
  \begin{equation*}
    \psi_\rho(\varphi(a))\rho_E(f)=\rho_A(a)\rho_E(f)
  \end{equation*}
  and hence
  \begin{equation*}
    \bigl(\psi_\rho(\varphi(a))-\rho_A(a)\bigr)\rho_E(f)=0.
  \end{equation*}
  By Proposition~\ref{2.7} and the definition of $\I_E$, for any $b\in\ker\varphi$ we have
  \begin{equation*}
    \psi_\rho(\varphi(a))\rho_A(b)
    =0
    =\rho_A(0)
    =\rho_A(ab)
    =\rho_A(a)\rho_A(b)
  \end{equation*}
  and hence
  \begin{equation*}
    \bigl(\psi_\rho(\varphi(a))-\rho_A(a)\bigr)\rho_A(b)=0.
  \end{equation*}
  Therefore, By Lemma~\ref{5.5} we have $\psi_\rho(\varphi(a))-\rho_A(a)=0$, i.e., $\psi_\rho(\varphi(a))=\rho_A(a)$.
  Since this holds for any $a\in\I_E$, the representation $\rho$ is covariant.
\end{proof}

\begin{proof}[Proof of Theorem \ref{5.1}]
  Proposition \ref{5.3} asserts that $\rho$ is a representation of $(A,E)$.
  Proposition \ref{5.4} asserts that $\rho$ is injective as a representation of a C*-correspondence.
  Proposition \ref{5.6} asserts that $\rho$ is covariant.
\end{proof}

\begin{proof}[Proof of Theorem \ref{5.2}]
  By Theorem \ref{5.1}, $\rho$ is an injective covariant representation.
  Therefore, by Proposition \ref{2.12}, $\hat\rho$ is universal as a covariant representation.
\end{proof}

\subsection{Construction of $\rho_n$ and $E_n$}\label{subsection 5.2}

In this subsection and subsequent subsections, we assume the following additional condition (IV):
\begin{itemize}
  \item[(IV)] All of $\gamma_1$, \ldots, $\gamma_N$ are embeddings.
\end{itemize}

Let $I := \{1, \ldots, N\}$ be the index set.
For $\boldsymbol k = (k_n, \ldots, k_1)\in I^n$ and a point $x$ in $X$, let $\boldsymbol\gamma_{\boldsymbol k}(x)$ denote the $\gamma$-chain $(\gamma_{k_n,\ldots,k_1}(x), \ldots, \gamma_{k_1}(x), x)$.
We define $X_n$, $\rho_n$ and $E_n$ as below for each $n\in\N_0$.
\begin{itemize}
  \item Let $X_n$ be the topological space consisting of all $\gamma$-chains of length $n$. This space is Hausdorff and compact.
  \item Let $\rho_n\colon C(X_n)\to M_\X$ be the map defined by 
  \begin{equation*}
    \rho_n(f)\e_x:=\sum_{\boldsymbol k\in I^n}\e_{\gamma_{\boldsymbol k}(x)}f(\boldsymbol\gamma_{\boldsymbol k}(x))
  \end{equation*}
  for $f\in C(X_n)$ and $x\in\X$. This is complex-linear.
  \item Let $E_n$ be the range of $\rho_n$.
\end{itemize}

\begin{remark}
  The map $\rho_0$ coincides with $\rho_A$, and $\rho_1$ coincides with $\rho_E$.
\end{remark}

\begin{lemma}\label{5.2.1}
  For $n\in\N_0$, the map $\rho_n$ is bi-Lipschitz as a linear operator from $C(X_n)$ with the supremum norm to $M_\X$ with the operator norm.
\end{lemma}

\begin{proof}
  We show that $\rho_n$ is Lipschitz, i.e. bounded.
  Take any $f\in C(X_n)$ and $\boldsymbol v=\sum_{x\in\X}\e_x v_x, \boldsymbol w=\sum_{x\in\X}\e_x w_x\in\H_\X$.
  We have
  \begin{equation*}
    \rho_n(f)\boldsymbol w
    =\rho_n(f)\sum_{x\in\X}\e_x w_x
    =\sum_{x\in\X}\rho_n(f)\e_x w_x
    =\sum_{x\in\X}\sum_{\boldsymbol k\in I^n}\e_{\gamma_{\boldsymbol k}(x)}f(\boldsymbol\gamma_{\boldsymbol k}(x))w_x;
  \end{equation*}
  and hence
  \begin{align*}
    &\abs{\inn{\boldsymbol v}{\rho_n(f)\boldsymbol w}}\\
    &\leq\sum_{x\in\X}\sum_{\boldsymbol k\in I^n}\abs{v_{\gamma_{\boldsymbol k}(x)}}\abs{f(\boldsymbol\gamma_{\boldsymbol k}(x))}\abs{w_x}\\
    &\leq\biggl(\sum_{x\in\X}\sum_{\boldsymbol k\in I^n}\abs{v_{\gamma_{\boldsymbol k}(x)}}^2\abs{f(\boldsymbol\gamma_{\boldsymbol k}(x))}^2\biggr)^{1/2}\biggl(\sum_{x\in\X}\sum_{\boldsymbol k\in I^n}\abs{w_x}^2\biggr)^{1/2}\\
    &\leq\biggl(\sum_{x\in\X}\sum_{\boldsymbol k\in I^n}\abs{v_{\gamma_{\boldsymbol k}(x)}}^2\sup\abs{f}^2\biggr)^{1/2}\biggl(\sum_{x\in\X}\sum_{\boldsymbol k\in I^n}\abs{w_x}^2\biggr)^{1/2}\\
    &=\biggl(\sum_{x\in\X}\sum_{\boldsymbol k\in I^n}\abs{v_{\gamma_{\boldsymbol k}(x)}}^2\biggr)^{1/2}\cdot\sup\abs{f}\cdot N^{n/2}\cdot\biggl(\sum_{x\in\X}\abs{w_x}^2\biggr)^{1/2}\\
    &\leq\norm{\boldsymbol v}\cdot\sup\abs{f}\cdot N^{n/2}\cdot\norm{\boldsymbol w}.
  \end{align*}
  Therefore, the inequality $\norm{\rho_n(f)}\leq\sup\abs{f}\cdot N^{n/2}$ holds for all $f\in C(X_n)$ and hence $\rho_n$ is bounded.

  We show that $\rho_n$ is bi-Lipschitz.
  For any $x\in\X$ and $\boldsymbol k\in I^n$, we have
  \begin{align*}
    \norm{\rho_n(f)}
    &=\norm{\rho_n(f)}\norm{\e_x}
    \geq\norm{\rho_n(f)\e_x}
    =\NORM\Big{\sum_{\boldsymbol j\in I^n}\e_{\gamma_{\boldsymbol j}(x)}f(\boldsymbol\gamma_{\boldsymbol j}(x))}\\
    &\geq\norm{\e_{\gamma_{\boldsymbol k}(x)}f(\boldsymbol\gamma_{\boldsymbol k}(x))}
    =\abs{f(\boldsymbol\gamma_{\boldsymbol k}(x))}.
  \end{align*}
  Since the set $\set{\boldsymbol\gamma_{\boldsymbol k}(x)}{x\in\X\text{ and }\boldsymbol k\in I^n}$ is dense in $X_n$, we have $\norm{\rho_n(f)}\geq\sup\abs{f}$.
  Therefore, $\rho_n$ is bi-Lipschitz.
\end{proof}

\begin{corollary}\label{5.2.2}
  The sets $E_n$ are norm-closed linear subspaces of $M_\X$.
\end{corollary}

\begin{proof}
  Fix $n\in\N_0$.
  Since the set $E_n$ is defined as the range of the complex-linear map $\rho_n$, 
  it is a complex-linear subspace of $M_\X$.
  Moreover, since the domain $C(X_n)$ of $\rho_n$ is norm-complete and $\rho_n$ is bi-Lipschitz (Lemma~\ref{5.2.1}), the range $E_n$ of $\rho_n$ is norm-closed.
\end{proof}

\begin{lemma}\label{5.2.3}
  The equation 
  \begin{equation*}
    \rho_m(f)\rho_n(g)=\rho_{m+n}(f\boxdot g)
  \end{equation*}
  holds for all $m,n\in\N_0$, $f\in C(X_m)$ and $g\in C(X_n)$.
\end{lemma}

\begin{proof}
  For any $x\in\X$, we have
  \begin{align*}
    \rho_m(f)\rho_n(g)\e_x
    &=\rho_m(f)\sum_{\boldsymbol k\in I^n}\e_{\gamma_{\boldsymbol k}(x)}g(\boldsymbol\gamma_{\boldsymbol k}(x))\\
    &=\sum_{\boldsymbol k\in I^n}\rho_m(f)\e_{\gamma_{\boldsymbol k}(x)}g(\boldsymbol\gamma_{\boldsymbol k}(x))\\
    &=\sum_{\boldsymbol k\in I^n}\sum_{\boldsymbol j\in I^m}\e_{\gamma_{\boldsymbol j}(\gamma_{\boldsymbol k}(x))} f(\boldsymbol\gamma_{\boldsymbol j}(\gamma_{\boldsymbol k}(x)))g(\boldsymbol\gamma_{\boldsymbol k}(x))\\
    &=\sum_{\boldsymbol l\in I^{m+n}}\e_{\gamma_{\boldsymbol l}(x)}(f\boxdot g)(\boldsymbol\gamma_{\boldsymbol l}(x))\\
    &=\rho_{m+n}(f\boxdot g)\e_x.
  \end{align*}
  Therefore, the equation $\rho_m(f)\rho_n(g)=\rho_{m+n}(f\boxdot g)$ holds.
\end{proof}

\begin{proposition}\label{5.2.4}
  Let $m,n\in\N_0$.
  The norm-closed additive subgroup of $M_\X$ generated by ${E_m}{E_n}$ coincides with $E_{m+n}$.
\end{proposition}

\begin{remark}
  By this proposition, the following equation is satisfied for all $m,n\in\N_0$:
  \begin{equation*}
    \cspan({E_m}{E_n})={E_{m+n}}.
  \end{equation*}
\end{remark}

\begin{proof}
  By Lemma~\ref{5.2.3}, we have 
  \begin{equation*}
    E_m E_n
    =\rho_{m+n}(C(X_m)\boxdot C(X_n)).
  \end{equation*}
  Since the linearity of $\rho_{m+n}$ implies preservation of additive subgroups, we have 
  \begin{equation*}
    \asg{E_m E_n}
    =\rho_{m+n}(\asg{C(X_m)\boxdot C(X_n)}).
  \end{equation*}
  Furthermore, as the map $\rho_{m+n}$ is bi-Lipschitz (Lemma~\ref{5.2.1}), it also preserves norm-closures: 
  \begin{equation*}
    \casg{E_m E_n}
    =\rho_{m+n}(\casg{C(X_m)\boxdot C(X_n)}).
  \end{equation*}
  By the assumption (IV) and Proposition~\ref{3.15}, we know that $\casg{C(X_m)\boxdot C(X_n)}=C(X_{m+n})$.
  Therefore, 
  \begin{equation*}
    \casg{E_m E_n}
    =E_{m+n}.
  \end{equation*}
  Hence, we conclude that the norm-closed additive subgroup generated by $E_m E_n$ coincides with $E_{m+n}$.
\end{proof}

\begin{corollary}\label{5.2.5}
  All of the $E_n$ are contained in $C^*(\rho)$.
\end{corollary}

\begin{proof}
  $E_0=\rho_A(A)$ and $E_1=\rho_E(E)$ are contained in $C^*(\rho)$.
  If $E_n$ is contained in $C^*(\rho)$, 
  then $E_{n+1}$ is also contained, 
  since $E_{n+1}=\cspan({E_1}{E_n})$ holds by Proposition~\ref{5.2.4}.
  Therefore, by induction, all of $E_n$ are contained in $C^*(\rho)$.
\end{proof}

\begin{proposition}\label{5.2.6}
  The inclusion $E_n^*E_n\subseteq E_0$ holds for all $n\in\N_0$.
\end{proposition}

\begin{proof}
  We prove the statement by induction on $n$.

  \textbf{The case of $n = 0$.} 
  Since $\rho_0 = \rho_A$ is a $*$-homomorphism, its range $E_0$ is closed under multiplication and involution.
  Hence, $E_0^*E_0\subseteq E_0$ holds.

  \textbf{The case of $n = 1$.} 
  For any $a, b\in C(X_1) = E$, we have 
  \begin{equation*}
    \rho_1(a)^*\rho_1(b)
    = \rho_E(a)^*\rho_E(b)
    = \rho_A(\inn a b)
    = \rho_0(\inn a b)
    \in E_0.
  \end{equation*}
  Thus, $E_1^*E_1\subseteq E_0$ follows.

  \textbf{Induction step.} 
  Let $n\in\N$.
  Assume that $E_n^*E_n\subseteq E_0$.
  By Proposition~\ref{5.2.4}, we have $E_{1+n} = \casg{E_1 E_n}$ and $E_0 E_n \subseteq E_n$.
  So we have
  \begin{gather*}
    E_{1+n}^*E_{1+n}
    = \Bigl(\casg{E_1 E_n}\Bigr)^*\cdot\casg{E_1 E_n}
    = \casg{E_n^* E_1^* E_1 E_n}\\
    \subseteq \casg{E_n^* E_0 E_n}
    \subseteq \casg{E_n^* E_n}
    \subseteq \casg{E_0}
    = E_0.
  \end{gather*}
  Thus, $E_{1+n}^*E_{1+n}\subseteq E_0$ follows, completing the induction.
\end{proof}

\subsection{Construction of $\rho_{m,n}$ and $E_{m,n}$}

We define $X_{m,n}$, $\rho_{m,n}$ and $E_{m,n}$ as follows for each $m,n\in\N_0$.
\begin{itemize}
  \item Let $X_{m,n}$ be the topological space consisting of all $\gamma$-bichains of length $(m,n)$. This is Hausdorff and compact.
  \item The map $\rho_{m,n}\colon C(X_{m,n})\to M_\X$ is defined as follows:
    for $f\in C(X_{m,n})$ and $y\in\X$, if $y\in\gamma^n(\X)$, 
    there exists a unique index sequence $\boldsymbol k$ of length $n$ and $x\in\X$ such that $y=\gamma_{\boldsymbol k}(x)$, and we define 
    \begin{equation*}
      \rho_{m,n}(f)\e_y
      :=\sum_{\boldsymbol j\in I^m}\e_{\gamma_{\boldsymbol j}(x)}
      f(\boldsymbol\gamma_{\boldsymbol j}(x), \boldsymbol\gamma_{\boldsymbol k}(x)).
    \end{equation*}
    If $y\not\in\gamma^n(\X)$, we define $\rho_{m,n}(f)\e_y = 0$.
  \item Let $E_{m,n}$ be the range of $\rho_{m,n}$.
\end{itemize}

\begin{fact}[\cite{I}, Lemma 5.15]\label{5.3.1}
  Let $m,n\in\N$ with $m\leq n$ and $\alpha$ be a complex $m\times n$ matrix.
  Then, the inequality
  \begin{equation*}
    (\text{the operator norm of $\alpha$})\leq n\times (\text{the max norm of $\alpha$})
  \end{equation*}
  holds.
  Here, the max norm of a matrix is the maximum of the absolute values of entries of the matrix.
\end{fact}

\begin{lemma}\label{5.3.2}
  For $m,n\in\N_{0}$, the map $\rho_{m,n}$ is bi-Lipschitz as a linear operator from $C(X_{m,n})$ with supremum norm into $M_\X$ with the operator norm.
\end{lemma}

\begin{proof}
  \textbf{Step 1: Lipschitz property.}
  Take any $f\in C(X_{m,n})$ and $\boldsymbol v,\boldsymbol w\in\H_\X$.
  For each $x\in\X$, define vectors $\boldsymbol v_x\in\C^{N^m}$ and $\boldsymbol w_x\in\C^{N^n}$ as follows:
  \begin{equation*}
    \boldsymbol v_x
    :=\sum_{\boldsymbol j\in I^m}\e_{\boldsymbol j}\inn{\e_{\gamma_{\boldsymbol j}(x)}}{\boldsymbol v},\quad
    \boldsymbol w_x
    :=\sum_{\boldsymbol k\in I^n}\e_{\boldsymbol k}\inn{\e_{\gamma_{\boldsymbol k}(x)}}{\boldsymbol w}.
  \end{equation*}
  Let $\alpha_x\colon\C^{N^n}\to\C^{N^m}$ be the $N^m\times N^n$ matrix defined by 
  \begin{equation*}
    \inn{\e_{\boldsymbol j}}{\alpha_x\e_{\boldsymbol k}}
    :=f(\boldsymbol\gamma_{\boldsymbol j}(x), \boldsymbol\gamma_{\boldsymbol k}(x)).
  \end{equation*}
  By Fact~\ref{5.3.1}, we have $\norm{\alpha_x}\leq N^{\max(m,n)}\sup\abs{f}$.
  So we have
  \begin{align*}
    \inn{\boldsymbol v}{\rho_{m,n}(f)\boldsymbol w}
    &=\inn{\boldsymbol v}{\rho_{m,n}(f)\sum_{y\in\X}\e_y\inn{\e_y}{\boldsymbol w}}\\
    &=\sum_{y\in\X}\inn{\boldsymbol v}{\rho_{m,n}(f)\e_y}\inn{\e_y}{\boldsymbol w}\\
    &=\sum_{\boldsymbol k\in I^n}\sum_{x\in\X}\inn{\boldsymbol v}{\rho_{m,n}(f)\e_{\gamma_{\boldsymbol k}(x)}}\inn{\e_{\gamma_{\boldsymbol k}(x)}}{\boldsymbol w}\\
    &=\sum_{\boldsymbol k\in I^n}\sum_{x\in\X}\sum_{\boldsymbol j\in I^m}\inn{\boldsymbol v}{\e_{\gamma_{\boldsymbol j}(x)}}f(\boldsymbol\gamma_{\boldsymbol j}(x), \boldsymbol\gamma_{\boldsymbol k}(x))\inn{\e_{\gamma_{\boldsymbol k}(x)}}{\boldsymbol w}\\
    &=\sum_{x\in\X}\inn{\boldsymbol v_x}{\alpha_x\boldsymbol w_x};
  \end{align*}
  and hence
  \begin{align*}
    \abs{\inn{\boldsymbol v}{\rho_{m,n}(f)\boldsymbol w}}
    &\leq\sum_{x\in\X}\abs{\inn{\boldsymbol v_x}{\alpha_x\boldsymbol w_x}}\\
    &\leq\sum_{x\in\X}\norm{\boldsymbol v_x}\norm{\alpha_x}\norm{\boldsymbol w_x}\\
    &\leq\sum_{x\in\X}\norm{\boldsymbol v_x}\cdot N^{\max(m,n)}\sup\abs{f}\cdot\norm{\boldsymbol w_x}\\
    &=N^{\max(m,n)}\sup\abs{f}\sum_{x\in\X}\norm{\boldsymbol v_x}\norm{\boldsymbol w_x}\\
    &\leq N^{\max(m,n)}\sup\abs{f}\biggl(\sum_{x\in\X}\norm{\boldsymbol v_x}^2\biggr)^{1/2}\biggl(\sum_{x\in\X}\norm{\boldsymbol w_x}^2\biggr)^{1/2}\\
    &\leq N^{\max(m,n)}\sup\abs{f}\norm{\boldsymbol v}\norm{\boldsymbol w}.
  \end{align*}
  Therefore, 
  \begin{equation*}
    \norm{\rho_{m,n}(f)}\leq N^{\max(m,n)}\sup\abs{f}
  \end{equation*}
  holds for all $f\in C(X_{m,n})$.
  Thus, we have shown the Lipschitz property for $\rho_{m,n}$.

  \textbf{Step 2: reverse Lipschitz property.}
  Take any $f\in C(X_{m,n})$.
  For any $x\in\X$ and $\boldsymbol j\in I^m$, $\boldsymbol k\in I^n$, we have
  \begin{gather*}
    \norm{\rho_{m,n}(f)}
    =\norm{\rho_{m,n}(f)}\norm{\e_{\gamma_{\boldsymbol k}(x)}}
    \geq\norm{\rho_{m,n}(f)\e_{\gamma_{\boldsymbol k}(x)}}\\
    =\sum_{\boldsymbol l\in I^m}\norm{\e_{\gamma_{\boldsymbol l}(x)}}\abs{f(\boldsymbol\gamma_{\boldsymbol l}(x), \boldsymbol\gamma_{\boldsymbol k}(x))}
    \geq\abs{f(\boldsymbol\gamma_{\boldsymbol j}(x), \boldsymbol\gamma_{\boldsymbol k}(x))}.
  \end{gather*}
  Hence, letting 
  \begin{equation*}
    \X_{m,n}
    :=\set
    {(\boldsymbol\gamma_{\boldsymbol j}(x), \boldsymbol\gamma_{\boldsymbol k}(x))}
    {x\in\X,\ \boldsymbol j\in I^m,\ \boldsymbol k\in I^n},
  \end{equation*}
  we have the inequality 
  \begin{equation*}
    \norm{\rho_{m,n}(f)}
    \geq\sup_{\boldsymbol z\in\X_{m,n}}\abs{f(\boldsymbol z)}.
  \end{equation*}
  We note that 
  \begin{equation*}
    X_{m,n}
    =\set
    {(\boldsymbol\gamma_{\boldsymbol j}(x), \boldsymbol\gamma_{\boldsymbol k}(x))}
    {x\in X,\ \boldsymbol j\in I^m,\ \boldsymbol k\in I^n}.
  \end{equation*}
  Since the maps $\gamma_1,\ldots,\gamma_N$ are embeddings and $\X$ is dense in $X$, the set $\X_{m,n}$ is dense in $X_{m,n}$.
  So we conclude $\norm{\rho_{m,n}(f)}\geq\sup\abs{f}$.
  Thus, we have shown the reverse Lipschitz property for $\rho_{m,n}$.
\end{proof}

\begin{corollary}\label{5.3.3}
  The sets $E_{m,n}$ are norm-closed linear subspaces of $M_\X$.
\end{corollary}

\begin{proof}
  Fix $m,n\in\N_0$.
  Since the set $E_{m,n}$ is defined as the range of the complex-linear map $\rho_{m,n}$, it is a complex-linear subspace of $M_\X$.
  Moreover, since the domain $C(X_{m,n})$ of $\rho_{m,n}$ is norm-complete and $\rho_{m,n}$ is bi-Lipschitz (Lemma~\ref{5.3.2}), the range $E_{m,n}$ of $\rho_{m,n}$ is norm-closed.
\end{proof}

\begin{lemma}\label{5.3.4}
  The equation 
  \begin{equation*}
    \rho_m(f)\rho_n(g)^*=\rho_{m,n}(f\boxtimes g)
  \end{equation*}
  holds for all $m,n\in\N_0$, $f\in C(X_m)$ and $g\in C(X_n)$.
\end{lemma}

\begin{proof}
  Take any points $x,y\in\X$.
  We have
  \begin{align*}
    &\inn{\e_x}{\rho_m(f)\rho_n(g)^*\e_y}\\
    &=\inn{\e_x}{\rho_m(f)\sum_{z\in\X}\e_z\inn{\e_z}{\rho_n(g)^*\e_y}}\\
    &=\sum_{z\in\X}\inn{\e_x}{\rho_m(f)\e_z}\overline{\inn{\e_y}{\rho_n(g)\e_z}}\\
    &=\sum_{z\in\X}
    \inn{\e_x}{\sum_{\boldsymbol j\in I^m}\e_{\gamma_{\boldsymbol j}(z)}f(\boldsymbol\gamma_{\boldsymbol j}(z))}
    \overline{\inn{\e_y}{\sum_{\boldsymbol k\in I^n}\e_{\gamma_{\boldsymbol k}(z)}g(\boldsymbol\gamma_{\boldsymbol k}(z))}}\\
    &=\sum_{z\in\X}\sum_{\boldsymbol j\in I^m}\sum_{\boldsymbol k\in I^n}
    \inn{\e_x}{\e_{\gamma_{\boldsymbol j}(z)}}
    f(\boldsymbol\gamma_{\boldsymbol j}(z))
    \overline{g(\boldsymbol\gamma_{\boldsymbol k}(z))}
    \inn{\e_{\gamma_{\boldsymbol k}(z)}}{\e_y}.
  \end{align*}
  If there exist $\boldsymbol j\in I^m$, $\boldsymbol k\in I^n$ and $z\in\X$ satisfying $x=\gamma_{\boldsymbol j}(z)$ and $y=\gamma_{\boldsymbol k}(z)$, the indices $\boldsymbol j$ and $\boldsymbol k$ are uniquely determined and we have 
  \begin{equation*}
    \inn{\e_x}{\rho_m(f)\rho_n(g)^*\e_y}=f(\boldsymbol\gamma_{\boldsymbol j}(z))\overline{g(\boldsymbol\gamma_{\boldsymbol k}(z))};
  \end{equation*}
  and otherwise we have $\inn{\e_x}{\rho_m(f)\rho_n(g)^*\e_y}=0$.
  Therefore, we have 
  \begin{equation*}
    \inn{\e_x}{\rho_m(f)\rho_n(g)^*\e_y}
    =\inn{\e_x}{\rho_{m,n}(f\boxtimes g)\e_y}.
  \end{equation*}
  Since this equality holds for the orthonormal system $\{\e_x\}_{x\in\X}$ of $\H_\X$, it follows that
  \begin{equation*}
    \rho_m(f)\rho_n(g)^*=\rho_{m,n}(f\boxtimes g).\qedhere
  \end{equation*}
\end{proof}

\begin{proposition}\label{5.3.5}
  Let $m,n\in\N_0$.
  The norm-closed additive subgroup of $M_\X$ generated by $E_m E_n^*$ coincides with $E_{m,n}$.
\end{proposition}

\begin{remark}
  By this proposition, the following equation is satisfied for all $m,n\in\N_0$:
  \begin{equation*}
    \cspan(E_m E_n^*)=E_{m,n}.
  \end{equation*}
\end{remark}

\begin{proof}
  By Lemma~\ref{5.3.4}, we have 
  \begin{equation*}
    E_m E_n^*
    =\rho_{m,n}(C(X_m)\boxtimes C(X_n)).
  \end{equation*}
  Since the linearity of $\rho_{m,n}$ implies preservation of additive subgroups, we have 
  \begin{equation*}
    \asg{E_m E_n^*}
    =\rho_{m,n}(\asg{C(X_m)\boxtimes C(X_n)}).
  \end{equation*}
  Furthermore, as the map $\rho_{m,n}$ is bi-Lipschitz (Lemma~\ref{5.3.2}), it also preserves norm-closures:
  \begin{equation*}
    \casg{E_m E_n^*}
    =\rho_{m,n}(\casg{C(X_m)\boxtimes C(X_n)}).
  \end{equation*}
  By Proposition~\ref{3.21}, we know that $\casg{C(X_m)\boxtimes C(X_n)}=C(X_{m,n})$.
  Therefore, 
  \begin{equation*}
    \casg{E_m E_n^*}
    =E_{m,n}.
  \end{equation*}
  Hence, we conclude that the norm-closed additive subgroup generated by $E_m E_n^*$ coincides with $E_{m,n}$.
\end{proof}

\begin{corollary}\label{5.3.6}
  All of the $E_{m,n}$ are contained in $C^*(\rho)$.
\end{corollary}

\begin{proof}
  Let $m,n\in\N_0$.
  By Corollary~\ref{5.2.5}, both $E_m$ and $E_n$ are contained in $C^*(\rho)$.
  By Proposition~\ref{5.3.5}, we have the equation $\casg{E_m E_n^*}=E_{m,n}$, which implies that $E_{m,n}$ is also contained in $C^*(\rho)$.
\end{proof}

\begin{proposition}\label{5.3.7}
  Let $m,n,p,q\in\N_0$.
  \begin{itemize}
    \item If $n\geq p$, then $E_{m,n}E_{p,q}\subseteq E_{m,n-p+q}$.
    \item If $n\leq p$, then $E_{m,n}E_{p,q}\subseteq E_{m-n+p,q}$.
  \end{itemize}
\end{proposition}

\begin{proof}
  We assume that $n\geq p$.
  By Proposition~\ref{5.2.6}, we have $E_p^* E_p\subseteq E_0 = E_0^*$.
  Hence, by Proposition~\ref{5.2.4} and Corollary~\ref{5.2.2}, we obtain 
  \begin{gather*}
    E_n^* E_p 
    = \Bigl(\casg{E_p E_{n-p}}\Bigr)^*\cdot E_p 
    = \casg{E_{n-p}^* E_p^*}\cdot E_p
    \subseteq \casg{E_{n-p}^* E_p^* E_p}\\
    \subseteq \casg{E_{n-p}^* E_0^*}
    = \Bigl(\casg{E_0 E_{n-p}}\Bigr)^*
    \subseteq \Bigl(\casg{E_{n-p}}\Bigr)^*
    = E_{n-p}^*.
  \end{gather*}
  Therefore, we have 
  \begin{equation*}
    \asg{E_m E_n^*}\asg{E_p E_q^*}
    \subseteq\asg{E_m E_n^* E_p E_q^*}
    \subseteq\asg{E_m E_{n-p}^* E_q^*}
    \subseteq\asg{E_m E_{n-p+q}^*}.
  \end{equation*}
  By Proposition~\ref{5.3.5}, it follows that 
  \begin{equation*}
    E_{m,n}E_{p,q}
    =\casg{E_m E_n^*}\cdot\casg{E_p E_q^*}
    \subseteq\casg{E_m E_{n-p+q}^*}
    =E_{m,n-p+q}.
  \end{equation*}
  The case $n\leq p$ follows by a similar argument.
\end{proof}

\begin{proposition}\label{5.3.8}
  The equation $E_{m,n}^*=E_{n,m}$ holds for all $m,n\in\N_0$.
\end{proposition}

\begin{proof}
  By Proposition~\ref{5.3.5}, we have 
  \begin{equation*}
    E_{m,n}^*
    =\Bigl(\casg{E_m E_n^*}\Bigr)^*
    =\casg{E_n E_m^*}
    =E_{n,m}.
    \qedhere
  \end{equation*}
\end{proof}

\subsection{The subspace $B_k$}

For each $k\in\Z$, we define closed subspaces $B_k$ of $C^*(\rho)$ as follows:
\begin{equation*}
  B_k:=\casg{\bigcup_{\substack{m,n\in\N_0\\m-n=k}}E_{m,n}}.
\end{equation*}

\begin{remark}
  The range $\rho_A(A)=E_{0,0}$ is contained in $B_0$.
\end{remark}

\begin{proposition}\label{5.4.1}
  The inclusion ${B_k}{B_l}\subseteq B_{k+l}$ holds for all $k,l\in\Z$.
\end{proposition}

\begin{proof}
  By Proposition~\ref{5.3.7}, we have 
  \begin{equation*}
    E_{m,n}E_{p,q}
    \subseteq\bigcup_{\substack{s,t\in\N_0\\s-t=k+l}}E_{s,t}
  \end{equation*}
  for all $m,n,p,q\in\N_0$ with $m-n=k$ and $p-q=l$.
  Therefore, we obtain 
  \begin{equation*}
    \asg{\bigcup_{\substack{m,n\in\N_0\\m-n=k}}E_{m,n}}
    \asg{\bigcup_{\substack{p,q\in\N_0\\p-q=l}}E_{p,q}}
    \subseteq\asg{\bigcup_{\substack{s,t\in\N_0\\s-t=k+l}}E_{s,t}}
  \end{equation*}
  Hence, we conclude that 
  \begin{equation*}
    B_k B_l \subseteq B_{k+l}.
    \qedhere
  \end{equation*}
\end{proof}

\begin{proposition}\label{5.4.2}
  The equation $B_k^*=B_{-k}$ holds for all $k\in\Z$.
\end{proposition}

\begin{proof}
  By Proposition~\ref{5.3.8}, we have 
  \begin{equation*}
    B_k^*
    =\biggl(\casg{\bigcup_{\substack{m,n\in\N_0\\m-n=k}}E_{m,n}}\biggr)^*
    =\casg{\bigcup_{\substack{m,n\in\N_0\\m-n=k}}E_{m,n}^*}
    =\casg{\bigcup_{\substack{m,n\in\N_0\\m-n=k}}E_{n,m}}
    =B_{-k}.
    \qedhere
  \end{equation*}
\end{proof}

\begin{corollary}\label{5.4.3}
  The following equation holds: 
  \begin{equation*}
    C^*(\rho)=\casg{\bigcup_{k\in\Z}B_k}.
  \end{equation*}
\end{corollary}

\begin{proof}
  Since the $B_k$ are defined as closed subspaces of $C^*(\rho)$, the inclusion $\supseteq$ holds.
  By Proposition~\ref{5.4.1} and Proposition~\ref{5.4.2}, $\asg{\bigcup_{k\in\Z}B_k}$ forms a $*$-subalgebra of $C^*(\rho)$.
  Note that $\rho_A(A)=E_{0,0}\subseteq B_0$ and $\rho_E(E)=E_{1,0}\subseteq B_1$.
  Since $C^*(\rho)$ is defined as the C*-subalgebra of $M_\X$ generated by $\rho_A(A)\cup\rho_E(E)$, the inclusion $\subseteq$ also holds.
\end{proof}

\begin{theorem}\label{5.4.4}
  The following equation holds: 
  \begin{equation*}
    C^*(\hat\rho) = \bigoplus_{k\in\Z} B_k \otimes \upsilon_k.
  \end{equation*}
  Here, $\upsilon_k$ denotes the element of the group C*-algebra $C^*(\Z)$ corresponding to $k\in\Z$.
\end{theorem}

\begin{proof}
  We show the inclusion $\subseteq$.
  Since $\rho_0 = \rho_A$, their ranges $E_0$ and $\rho_A(A)$ are equal.
  So we have 
  \begin{equation*}
    B_0 \otimes \upsilon_0
    \supseteq E_0 \otimes \upsilon_0
    = \rho_A(A) \otimes \upsilon_0
    = \hat\rho_A(A).
  \end{equation*}
  Similarly, since $\rho_1 = \rho_E$, their ranges $E_1$ and $\rho_E(E)$ are equal.
  So we have 
  \begin{equation*}
    B_1 \otimes \upsilon_1
    \supseteq E_1 \otimes \upsilon_1
    = \rho_E(E) \otimes \upsilon_1
    = \hat\rho_E(E).
  \end{equation*}
  Hence, the direct sum $\bigoplus_{k\in\Z}B_k\otimes\upsilon_k$ contains both $\hat\rho_A(A)$ and $\hat\rho_E(E)$.
  By Proposition~\ref{5.4.1} and Proposition~\ref{5.4.2}, this direct sum forms a closed $*$-subalgebra of $M_\X$.
  Since $C^*(\hat\rho)$ is defined as the C*-subalgebra of $M_\X\otimes C^*(\Z)$ generated by $\hat\rho_A(A)\cup\hat\rho_E(E)$, the inclusion $\subseteq$ holds.
  
  \medskip
  
  We show the inverse inclusion $\supseteq$.
  
  First, we show that $E_n\otimes\upsilon_n\subseteq C^*(\hat\rho)$ for all $n\in\N_0$.
  The sets $E_0\otimes\upsilon_0=\hat\rho_A(A)$ and $E_1\otimes\upsilon_1=\hat\rho_E(E)$ are contained in $C^*(\hat\rho)$.
  By Proposition~\ref{2.4.3} and Proposition~\ref{5.2.4}, we deduce 
  \begin{equation*}
    E_{n+1}\otimes\upsilon_{n+1}
    =\casg{E_n E_1}\otimes\upsilon_n\upsilon_1
    =\casg{(E_n\otimes\upsilon_n)(E_1\otimes\upsilon_1)}.
  \end{equation*}
  By induction, $E_n\otimes\upsilon_n$ is contained in $C^*(\hat\rho)$ for all $n\in\N_0$.

  Next, we show that $E_{m,n}\otimes\upsilon_{m-n}\subseteq C^*(\hat\rho)$ for all $m,n\in\N_0$.
  By Proposition~\ref{2.4.3} and Proposition~\ref{5.3.5}, we have 
  \begin{gather*}
    E_{m,n}\otimes\upsilon_{m-n}
    =\casg{E_m E_n^*}\otimes\upsilon_m\upsilon_n^*
    =\casg{(E_m\otimes\upsilon_m)(E_n\otimes\upsilon_n)^*}.
  \end{gather*}
  Thus, $E_{m,n}\otimes\upsilon_{m-n}$ is contained in $C^*(\hat\rho)$ for all $m,n\in\N_0$.
  
  By the definition of $B_k$ and Proposition~\ref{2.4.3}, we have 
  \begin{equation*}
    B_k \otimes \upsilon_k
    = \casg{\bigcup_{\substack{m,n\in\N_0\\m-n=k}} E_{m,n}} \otimes \upsilon_k
    = \casg{\bigcup_{\substack{m,n\in\N_0\\m-n=k}} E_{m,n} \otimes \upsilon_k}.
  \end{equation*}
  Thus, $B_k \otimes \upsilon_k$ is contained in $C^*(\hat\rho)$ for all $k \in \Z$; and so is their direct sum, that is, the inclusion $\supseteq$ holds.
\end{proof}

\subsection{Universality of $\rho$}

In this subsection, we assume the following condition (V):
\begin{itemize}
  \item[(V)] The set $\X$ is disjoint from the union 
    \begin{equation*}
      \bigcup_{\boldsymbol j\neq\boldsymbol k}\set{x\in X}{\gamma_{\boldsymbol j}(x)=\gamma_{\boldsymbol k}(x)},
    \end{equation*}
    where $\boldsymbol j$ and $\boldsymbol k$ range over all finite sequences of indices in $\{1,\ldots,N\}$.
\end{itemize}

By the universality of $\hat\rho$ (Theorem~\ref{5.2}), we obtain the natural $*$-homomorphism from $C^*(\hat\rho)$ to $C^*(\rho)$, denoted by $\nu$.
We define the subset $\X_{m,n}$ of $X\times X$ by 
\begin{equation*}
  \X_{m,n} := \set{(\gamma_{\boldsymbol j}(x),\gamma_{\boldsymbol k}(x))}{x\in\X,\ \boldsymbol j\in I^m,\ \boldsymbol k\in I^n},
\end{equation*}
for $m,n\in\N_0$.
For each $k\in\Z$, we define $\delta_k$, $\mathcal X_k$, $F_b(\mathcal X_k)$, and $\ent_k$ as follows:
\begin{itemize}
  \item Define a bounded linear functional $\delta_k\colon C^*(\Z)\to\C1_{M_\Z}\cong\C$ by $\delta_k(\alpha):=\delta_\Z(\alpha\upsilon_{-k})$.
  Here, $\delta_\Z$ is the faithful conditional expectation given by Fact~\ref{2.5}.
  \item Let $\mathcal X_k$ be the union of the sets $\X_{m,n}$ where the indices $m$ and $n$ range over $\N_0$ such that $m - n = k$.
  \item Let $F_b(\mathcal X_k)$ be the Banach space of all bounded functions on $\mathcal X_k$ with supremum norm.
  \item Let $\ent_k$ be the map from $C^*(\rho)$ to $F_b(\mathcal X_k)$ which maps $\alpha$ to the function on $\mathcal X_k$ whose value at $(x,y)\in\mathcal X_k$ is the $(x,y)$-entry $\inn{\e_x}{\alpha\e_y}$ of $\alpha$ at $(x,y)\in\mathcal X_k$.
\end{itemize}

\begin{lemma}\label{5.4.5}
  Let $m,n,p,q\in\N_0$ with $m-n\neq p-q$.
  Then, the sets $\X_{m,n}$ and $\X_{p,q}$ are disjoint.
\end{lemma}

\begin{proof}
  Assume that $\X_{m,n}\cap\X_{p,q}\neq\emptyset$. 
  Then, there exist $\boldsymbol j\in I^m$, $\boldsymbol k = (k_1,\ldots,k_n)\in I^n$, $\boldsymbol s = (s_1,\ldots,s_p)\in I^p$, $\boldsymbol t\in I^q$, and $x,y\in\X$ such that $(\gamma_{\boldsymbol j}(x),\gamma_{\boldsymbol k}(x)) = (\gamma_{\boldsymbol s}(y),\gamma_{\boldsymbol t}(y))$.
  Without loss of generality, we may assume $m\leq p$.
  Since $\X$ is completely $\gamma$-invariant and $\gamma_1(\X),\ldots,\gamma_N(\X)$ are pairwise disjoint, $\boldsymbol j$ must be an initial segment of $\boldsymbol s$.
  Let $\boldsymbol r := (s_{m+1}, \ldots, s_p)$.
  Since $\gamma_1,\ldots,\gamma_N$ are injective on $\X$, we have $x = \gamma_{\boldsymbol r}(y)$.
  Define $\boldsymbol l := (k_1, \ldots, k_n, s_{m+1}, \ldots, s_p)$.
  Then, 
  \begin{equation*}
    \gamma_{\boldsymbol l}(y) 
    = \gamma_{\boldsymbol k}(\gamma_{\boldsymbol r}(y)) 
    = \gamma_{\boldsymbol k}(x) 
    = \gamma_{\boldsymbol t}(y).
  \end{equation*}
  The length of $\boldsymbol l$ is $n - m + p$ and the length of $\boldsymbol t$ is $q$.
  Since these lengths are distinct, we must have $\boldsymbol l\neq\boldsymbol t$.
  Since $y\in\X$, by the assumption (V), it follows that  $\gamma_{\boldsymbol l}(y)\neq\gamma_{\boldsymbol t}(y)$, which is a contradiction.
\end{proof}

\begin{corollary}\label{5.4.6}
  The sets $\mathcal X_k$ are pairwise disjoint.
\end{corollary}

\begin{proof}
  This corollary immediately follows from Lemma~\ref{5.4.5}.
\end{proof}

\begin{lemma}\label{5.4.7}
  Let $m,n\in\N_0$, $\alpha\in E_{m,n}$, and $x,y\in\X$.
  If $x\rel\alpha y$, i.e., $\inn{\e_x}{\alpha\e_y}\neq 0$, then $(x,y)\in\X_{m,n}$.
\end{lemma}

\begin{proof}
  By the definition of $E_{m,n}$, there exists $f\in C(X_{m,n})$ such that $\alpha = \rho_{m,n}(f)$.
  Since $\alpha\e_y\neq\boldsymbol 0$, there exists a unique index sequence $\boldsymbol k$ of length $n$ and $z\in\X$ such that $y = \gamma_{\boldsymbol k}(z)$.
  Since $\inn{\e_x}{\alpha\e_y}\neq 0$, there exists a unique index sequence $\boldsymbol j$ of length $m$ such that $x = \gamma_{\boldsymbol j}(z)$.
  Thus, the pair $(x,y) = (\gamma_{\boldsymbol j}(z), \gamma_{\boldsymbol k}(z))$ is in $\X_{m,n}$.
\end{proof}

\begin{corollary}\label{5.4.8}
  Let $k\in\Z$, $\alpha\in B_k$, and $x,y\in\X$.
  If $x\rel\alpha y$, then $(x,y)\in\mathcal X_k$.
\end{corollary}

\begin{proof}
  This corollary immediately follows from Lemma~\ref{5.4.7}.
\end{proof}

\begin{lemma}\label{5.4.9}
  Let $k\in\Z$.
  The restriction of $\ent_k$ to $B_k$ is injective.
\end{lemma}

\begin{proof}
  Take any $\alpha\in\ker(\ent_k)\cap B_k$.
  Since $\ent_k(\alpha)=0$, the $(x,y)$-entry of $\alpha$ vanishes for any $(x,y)\in\mathcal X_k$.
  By Corollary~\ref{5.4.8}, the other entries also vanish.
  Therefore, all entries of $\alpha$ are zero; hence, $\alpha = 0$.
\end{proof}

\begin{lemma}\label{5.4.10}
  Let $k\in\Z$.
  The following diagram is commutative.
  \begin{equation*}
    \xymatrix@M=8pt{
      C^*(\hat\rho)
        \ar[r]^{\nu}
        \ar[d]_{\id\otimes\delta_k}
      & C^*(\rho)
        \ar[d]^{\mathrm{ent}_k}
      \\
      B_k
        \ar@{>->}[r]_{\mathrm{ent}_k}
      & F_b(\mathcal{X}_k)
    }
  \end{equation*}
\end{lemma}

\begin{proof}
  Let $l\in\Z$ and $\alpha\in B_l$.
  For any $(x,y)\in\mathcal X_k$, we have the following:
  \begin{equation*}
    \ent_k(\nu(\alpha\otimes\upsilon_l))(x,y)
    = \ent_k(\alpha)(x,y).
  \end{equation*}
  Next, we calculate $\ent_k(\id\otimes\delta_k(\alpha\otimes\upsilon_l))$:
  \begin{align*}
    \ent_k(\id\otimes\delta_k(\alpha\otimes\upsilon_l))(x,y)
    &= \ent_k(\id(\alpha)\cdot\delta_k(\upsilon_l))(x,y)\\
    &= \begin{cases}
      \ent_k(\alpha)(x,y) & \text{if}\ l = k\\
      0 & \text{otherwise}.
    \end{cases}
  \end{align*}
  By Corollary~\ref{5.4.8}, the $(x,y)$-entry of $\alpha$ is zero if $l \neq k$.
  Therefore, we have the equation
  \begin{equation*}
    \ent_k(\nu(\alpha\otimes\upsilon_l))(x,y)
    = \ent_k(\id\otimes\delta_k(\alpha\otimes\upsilon_l))(x,y).
  \end{equation*}
  This discussion shows that for any $l\in\Z$, the equation $\ent_k\circ\nu = \ent_k\circ (\id\otimes\delta_k)$ holds on $B_l\otimes\upsilon_l$.
  Thus, by Theorem~\ref{5.4.4}, the equation holds on the whole of $C^*(\hat\rho)$.
\end{proof}

\begin{lemma}\label{5.4.11}
  The $*$-homomorphism $\nu$ is an isometry.
\end{lemma}

\begin{proof}
  It suffices to show the injectivity of $\nu$.
  Take any $\xi\in\ker\nu$.
  By Lemma~\ref{5.4.10}, we have 
  \begin{equation*}
    \ent_k(\id\otimes\delta_n(\xi))
    = \ent_k(\nu(\xi))
    = \ent_k(0)
    = 0.
  \end{equation*}
  By Lemma~\ref{5.4.9}, we conclude that $\id\otimes\delta_n(\xi)=0$.
  This holds for all $k\in\Z$.
  Therefore, by applying Fact~\ref{2.4.4}, we conclude that $\xi = 0$.
\end{proof}

\begin{theorem}\label{5.4.12}
  The representation $\rho$ is universal as a covariant representation and $C^*(\rho)$ is isomorphic to $\O_\gamma$.
\end{theorem}

\begin{proof}
  It follows from Theorem~\ref{5.2} and Lemma~\ref{5.4.11}.
\end{proof}

\subsection{Masa in $C^*(\hat\rho)$}\label{subsection 5.5}

In this subsection and subsequent subsections, we assume the following additional conditions (VI) -- (VIII):
\begin{itemize}
  \item[(VI)] The iterated function system $\gamma$ satisfies the open set condition. That is, there exists a dense open set $U$ in $X$ such that the images $\gamma_1(U),\ldots,\gamma_N(U)$ are pairwise disjoint open set in $X$ and contained in $U$.
  \item[(VII)] The set $\X$ is contained in $U$.
  \item[(VIII)] The set $\X$ is comeager in $X$. (By the Baire category theorem, $\X$ is dense in $X$.)
\end{itemize}

Recall that $D_\X$ denotes the set of all diagonal $\X$-squared matrices.
For each $k\in\Z$, let $D_k:=B_k\cap D_\X$.
This is a closed subspace of $C^*(\rho)$.
Note that $D_0$ contains $\rho_A(A)$ and $D_0\otimes\upsilon_0$ contains $\hat\rho_A(A)$.
By Proposition~\ref{5.4.1} and Proposition~\ref{5.4.2}, the following properties hold for all $k,l\in\Z$:
\begin{equation*}
  {D_k}{D_l}\subseteq D_{k+l},\quad
  D_k^*=D_{-k}.
\end{equation*}

We define a closed subspace $\hat D$ of $C^*(\hat\rho)$ as follows:
\begin{equation*}
  \hat D:=\bigoplus_{k\in\Z}D_k\otimes\upsilon_k.
\end{equation*}
This $\hat D$ is a closed commutative $*$-subalgebra of $C^*(\hat\rho)$.

\begin{theorem}\label{5.5.1}
  $\hat D$ is a masa in $C^*(\hat\rho)$ that contains $\hat\rho_A(A)$.
\end{theorem}

\begin{proof}
  It suffices to prove that $\hat D$ is maximal as a commutative $*$-subalgebra.
  Assume that $\hat D$ is not maximal.
  Then, there exists a commutative $*$-subalgebra $\mathcal D$ of $C^*(\hat\rho)$ such that $\hat D\subsetneq\mathcal D$.
  Let $\alpha\in\mathcal D\setminus\hat D$, which can be decomposed in the form $\alpha=\sum_{k\in\Z}\alpha_k\otimes\upsilon_k$ for some $\alpha_k\in B_k$.
  Since $\alpha$ is not in $\hat D$, $\alpha_k$ is not in $D_k$ for some $k$.

  There exist two distinct points $x$ and $y$ in $\X$ such that $\inn{\e_x}{\alpha_k\e_y}$ is not zero.
  We can take a real-valued continuous function $f$ on $X$ such that $f(x)=0$ and $f(y)=1$.
  The $\X$-squared matrix $\rho_A(f)$ is self-adjoint and diagonal.
  It follows that 
  \begin{gather*}
    \inn{\e_x}{\alpha_k\rho_A(f)\e_y}
    =\inn{\e_x}{\alpha_k\e_yf(y)}
    =\inn{\e_x}{\alpha_k\e_y1}
    =\inn{\e_x}{\alpha_k\e_y}
    \neq 0;\\
    \intertext{and}
    \inn{\e_x}{\rho_A(f)\alpha_k\e_y}
    =\inn{\rho_A(f)\e_x}{\alpha_k\e_y}
    =\inn{\e_xf(x)}{\alpha_k\e_y}
    =\inn{\e_x0}{\alpha_k\e_y}
    =0.
  \end{gather*}
  Hence, $\alpha_k\rho_A(f)\neq\rho_A(f)\alpha_k$ holds.

  On the other hand, we have 
  \begin{align*}
    \alpha\hat\rho_A(f)
    &=\bigl(\ \sum_{k\in\Z}\alpha_k\otimes\upsilon_k\ \bigr)\bigl(\rho_A(f)\otimes\upsilon_0\bigr)
    =\sum_{k\in\Z}(\alpha_k\rho_A(f))\otimes\upsilon_k;\\
    \hat\rho_A(f)\alpha
    &=\bigl(\rho_A(f)\otimes\upsilon_0\bigr)\bigl(\ \sum_{k\in\Z}\alpha_k\otimes\upsilon_k\ \bigr)
    =\sum_{k\in\Z}(\rho_A(f)\alpha_k)\otimes\upsilon_k.
  \end{align*}
  Since $\hat\rho_A(f)\in D_0\otimes\upsilon_0\subseteq\hat D\subsetneq\mathcal D$, the operators $\alpha$ and $\hat\rho_A(f)$ commute.
  Thus, the equation $\alpha_k\rho_A(f)=\rho_A(f)\alpha_k$ holds.

  By the above discussion, a contradiction follows.
  Therefore, $\hat D$ is necessarily maximal.
\end{proof}

\begin{lemma}\label{5.5.10}
  Assume that $\gamma$ is essentially free.
  Let $\X'$ be the complement in $X$ of the completely $\gamma$-invariant subset of $X$ generated by the set 
  \begin{equation*}
    S := (X\setminus\X)\cup\bigcup_{\boldsymbol k\neq\boldsymbol l}\set{x\in X}{\gamma_{\boldsymbol k}(x)=\gamma_{\boldsymbol l}(x)},
  \end{equation*}
  where $\boldsymbol k$ and $\boldsymbol l$ range over all finite sequences of indices in $\{1,\ldots,N\}$.
  Then, $\X'$ satisfies the following conditions:
  \begin{itemize}
    \item[(I')] The set $\X'$ is dense in $X$ and completely invariant under $\gamma$.
    \item[(II')] The images $\gamma_1(\X'),\ldots,\gamma_N(\X')$ are pairwise disjoint.
    \item[(III')] The restriction $\gamma_k|_{\X'}$ is injective for all $k=1,\ldots,N$.
    \item[(V')] The set $\X$ is disjoint from the union 
    \begin{equation*}
      \bigcup_{\boldsymbol j\neq\boldsymbol k}\set{x\in X}{\gamma_{\boldsymbol j}(x)=\gamma_{\boldsymbol k}(x)},
    \end{equation*}
    where $\boldsymbol j$ and $\boldsymbol k$ range over all finite sequences of indices in $\{1,\ldots,N\}$.
  \end{itemize}
\end{lemma}

\begin{proof}
  Since the set $\X'$ is contained in $\X$, the conditions (II'), (III'), and (V') follow from the conditions (II), (III), and (V), respectively.
  Since $\X'$ is the complement of a completely $\gamma$-invariant set, $\X'$ is also completely $\gamma$-invariant.
  Since $\gamma$ is essentially free, by the assumption (VIII), $S$ is meager.
  By the assumption (VI) and Proposition~\ref{3.11}, $\X'$ is comeager in $X$.
  By the Baire Category Theorem, $\X'$ is dense in $X$.
  Therefore, the condition (I') holds.
\end{proof}

By Lemma~\ref{5.5.10}, if $\gamma$ is essentially free, we may assume that condition (V) holds by replacing $\X$ with $\X'$ if necessary. 
Therefore, from now on, whenever $\gamma$ is assumed to be essentially free, we will implicitly assume that condition (V) is satisfied.

\begin{lemma}\label{5.5.11}
  Assume that $\gamma$ is essentially free.
  If $k\in\Z\setminus\{0\}$, then all the diagonal entries of $\X$-squared matrices in $B_k$ vanish.
\end{lemma}

\begin{proof}
  We show that all the diagonal entries of $\X$-squared matrices in $E_{m,n}$ vanish for all distinct $m,n\in\N_0$.
  Let $m,n\in\N_0$ be distinct and let $\alpha\in E_{m,n}$.
  Let $y\in\X$.
  Assume that the $(y,y)$-entry of $\alpha$ does not vanish, i.e., $\inn{\e_y}{\alpha\e_y}\neq 0$.
  Since $\alpha\e_y$ is not zero, $y$ is in $\gamma^n(\X)$.
  Hence, there exists a unique index sequence $\boldsymbol k$ of length $n$ such that $y = \gamma_{\boldsymbol k}(x)$.
  Since $\inn{\e_y}{\alpha\e_y}\neq 0$, we have $y = \gamma_{\boldsymbol j}(x)$.
  Thus, we have $\gamma_{\boldsymbol j}(x) = \gamma_{\boldsymbol k}(x)$.
  By the condition (V), we have $x\not\in\X$.
  Since $\X$ is completely $\gamma$-invariant, we must have $y\not\in\X$, which leads to a contradiction.
  Therefore, the $(y,y)$-entry of $\alpha$ vanishes for all $y\in\X$.

  Let $k\in\Z\setminus\{0\}$.
  Since the function $\alpha\mapsto\inn{\e_y}{\alpha\e_y}$ is a bounded linear functional on $M_{\X}$, all the diagonal entries of $\X$-squared matrices in $B_k$ vanish.
\end{proof}

\begin{lemma}\label{5.5.2}
  Assume that $\gamma$ is essentially free.
  Then, $D_k=\{0\}$ holds for all $k\in\Z$ with $k\neq 0$.
\end{lemma}

\begin{proof}
  Let $k\in\Z\setminus\{0\}$.
  By Lemma~\ref{5.5.11}, all the diagonal entries of $\X$-squared matrices in $B_k$ vanish.
  Thus, $D_k = \{0\}$ holds.
\end{proof}

\begin{lemma}\label{5.5.3}
  Assume that $\gamma$ is not essentially free.
  Then, there exists a non-empty finite sequence $\boldsymbol k$ of indices $\{1,\ldots,N\}$ such that the set $\set{x\in X}{x=\gamma_{\boldsymbol k}(x)}$ has non-empty interior.
\end{lemma}

\begin{proof}
  Since $\gamma$ is not essentially free, there exist distinct two finite sequences $\boldsymbol j = (j_1, \ldots, j_m)$ and $\boldsymbol l = (l_1, \ldots, l_n)$ of indices in $\{1,\ldots,N\}$ such that the set $Y:=\set{x\in X}{\gamma_{\boldsymbol j}(x)=\gamma_{\boldsymbol l}(x)}$ has non-empty interior.
  Without loss of generality, we may assume $m \leq n$.
  Since $U$ is a dense open subset of $X$, the intersection $U\cap Y$ has also non-empty interior.
  Since the images $\gamma_1(U),\ldots,\gamma_N(U)$ are pairwise disjoint and contained in $U$, $\boldsymbol j$ is an initial segment of $\boldsymbol l$, i.e., $j_1 = l_1$, \dots, $j_m = l_m$.
  Since $\boldsymbol j$ and $\boldsymbol l$ are distinct, we have $m < n$.
  Hence, $\boldsymbol k := (l_{m+1}, \ldots, l_n)$ is not empty.
  By the injectivity of $\gamma_1,\ldots,\gamma_N$, the map $\gamma_{\boldsymbol j}$ is injective and $Y = \set{x\in X}{x = \gamma_{\boldsymbol k}(x)}$ holds.
\end{proof}

\begin{lemma}\label{5.5.4}
  Assume that $\gamma$ is not essentially free.
  Then, there exists a non-zero integer $k$ such that $D_k\neq\{0\}$.
\end{lemma}

\begin{proof}
  Since $\gamma$ is not essentially free, by Lemma~\ref{5.5.3}, there exists a non-empty finite sequence $\boldsymbol k$ of indices in $\{1,\ldots,N\}$ such that the set $Y:=\set{x\in X}{\gamma_{\boldsymbol k}(x) = x}$ has non-empty interior.
  Let $n$ be the length of $\boldsymbol k$.
  Since $U$ is a dense open subset of $X$, the intersection $U\cap Y$ has non-empty interior.
  The set $\boldsymbol Y := \set{(\boldsymbol\gamma_{\boldsymbol k}(x),x)}{x\in U\cap Y}$ is contained in $X_{n,0}$.
  Since the map $x\mapsto (\boldsymbol\gamma_{\boldsymbol k}(x),x)$ is an open embedding of $U$ into $X_{n,0}$, $\boldsymbol Y$ has non-empty interior in $X_{n,0}$.
  There exists a continuous function $f$ on $X_{n,0}$ that does not vanish on $\boldsymbol Y$ and whose support is contained in $\boldsymbol Y$.
  By the assumptions (VII) and (VIII), the $\X$-squared matrix $\rho_{n,0}(f)$ is non-zero.
  Since $\rho_{n,0}(f)$ is non-zero and diagonal, and belongs to $E_{n,0}$.
  Thus, $D_n \neq \{0\}$ holds.
\end{proof}

\begin{theorem}\label{5.5.5}
  The following three conditions are equivalent:
  \begin{enumerate}
    \item $\gamma$ is essentially free.
    \item $D_k = \{0\}$ holds for all $k\in\Z$ with $k\neq 0$.
    \item $\hat D = D_0\otimes\upsilon_0$.
  \end{enumerate}
\end{theorem}

\begin{remark}
  Since $\hat\rho_A(A) = \rho_A(A)\otimes\upsilon_0$ is contained in $D_0\otimes\upsilon_0$, $\gamma$ must be essentially free for $\hat\rho_A(A)\cong C(X)$ to be a masa of $C^*(\hat\rho)\cong\O_\gamma$.
\end{remark}

\begin{proof}
  The equivalence of (1) and (2) follows from Lemma~\ref{5.5.2} and Lemma~\ref{5.5.4}.
  The equivalence of (2) and (3) is clear.
\end{proof}

\begin{lemma}\label{5.5.6}
  Assume that the set $\gamma^n(X)$ is clopen in $X$ for all $n\in\N$ and there exists a continuous map $\sigma$ from $\gamma(X)$ to $X$ such that $\sigma\circ\gamma_k = \id_X$ holds for all $k\in\{1,\ldots,N\}$.
  Then, for all $\alpha\in B_0$, the function $x\mapsto\inn{\e_x}{\alpha\e_x}$ on $\X$ has a unique continuous extension to a function on $X$.
\end{lemma}

\begin{proof}
  For $n\in\N_0$, define a map $\boldsymbol\sigma_n\colon\gamma^n(X)\to X_n$ by 
  \begin{equation*}
    \boldsymbol\sigma_n(x) := (x,\sigma(x),\sigma^2(x),\ldots,\sigma^n(x)).
  \end{equation*}
  This is a homeomorphism.
  The following two statements for $x,y\in X$ are equivalent:
  \begin{itemize}
    \item There exists an index $k\in\{1,\ldots,N\}$ such that $x = \gamma_k(y)$.
    \item $x\in\gamma(X)$ and $\sigma(x) = y$.
  \end{itemize}
  Hence, the following equation holds for any $f\in C(X_{n,n})$ and $x\in\X$:
  \begin{equation*}
    \inn{\e_x}{\rho_{n,n}(f)\e_x}
    = \begin{cases}
      f(\boldsymbol\sigma_n(x),\boldsymbol\sigma_n(x)) & \text{if}\ x\in \gamma^n(X),\\
      0 & \text{otherwise}.
    \end{cases}
  \end{equation*}
  Since $\gamma^n(X)$ are clopen, the function $x\mapsto\inn{\e_x}{\rho_{n,n}(f)\e_x}$ on $\X$ has a unique continuous extension to a function on $X$.
  Therefore, for all $\alpha\in E_{n,n}$, the function $x\mapsto\inn{\e_x}{\alpha\e_x}$ on $\X$ has a unique continuous extension to a function on $X$.
  Define an operator $\Phi\colon\asg{\bigcup_{n\in\N_0}E_{n,n}}\to C(X)$ by $\Phi(\alpha)(x) := \inn{\e_x}{\alpha\e_x}$ for $f\in\asg{\bigcup_{n\in\N_0}E_{n,n}}$ and $x\in\X$.
  This $\Phi$ is bounded; hence has a unique continuous extension to $B_0$.
  For all $\alpha\in B_0$, the function $x\mapsto\inn{\e_x}{\alpha\e_x}$ on $\X$ has a unique continuous extension to a function on $X$.
\end{proof}

\begin{theorem}\label{5.5.7}
  Assume that the set $\gamma^n(X)$ is clopen in $X$ for all $n\in\N$ and there exists a continuous map $\sigma$ from $\gamma(X)$ to $X$ such that $\sigma\circ\gamma_k = \id_X$ holds for all $k\in\{1,\ldots,N\}$.
  Then, one has $\rho_A(A) = D_0$.
\end{theorem}

\begin{proof}
  By Lemma~\ref{5.5.6}, for any $\alpha\in D_0$, defining $a\in C(X)$ by $a(x):=\inn{\e_x}{\alpha\e_x}$ for $x\in\X$, we have $\alpha = \rho_A(a)$.
  Thus, $D_0\subseteq\rho_A(A)$ holds.
  The inverse inclusion is clear.
\end{proof}

\subsection{Approximate unit}

\begin{lemma}\label{5.19}
  The equation $\lim_\lambda(a_\lambda\cdot f)=f$ holds for all approximate unit $\{a_\lambda\}$ of $A$ and $f\in E$.
\end{lemma}

\begin{proof}
  Let $G$ be the graph of $\gamma$.
  For each $k \in I$, since the map $\gamma_k$ is an embedding and therefore injective, we define the continuous function $g_k$ on $\gamma_k(X)$ by $g_k(x) := f(x, \gamma_k^{-1}(x))$ for $x \in \gamma_k(X)$.
  By the Tietze extension theorem, we can extend $g_k$ to a continuous function on $X$, and we denote this extension by $f_k$.
  Recall that $E$ coincides with $C(G)$ as a set and the norm of $E$ is equivalent to the uniform norm.
  We have 
  \begin{align*}
    \norm{a_\lambda f - f}_{C(G)}
    &=\sup_{(x,y)\in G}\abs{a_\lambda(x)f(x,y)-f(x,y)}\\
    &=\max_{k\in I}\sup_{x\in\gamma_k(X)}\abs{a_\lambda(x)f(x,\gamma_k^{-1}(x))-f(x,\gamma_k^{-1}(x))}\\
    &=\max_{k\in I}\sup_{x\in\gamma_k(X)}\abs{a_\lambda(x)f_k(x)-f_k(x)}\\
    &\leq\max_{k\in I}\sup_{x\in X}\abs{a_\lambda(x)f_k(x)-f_k(x)}\\
    &=\max_{k\in I}\norm{a_\lambda f_k - f_k}_{C(X)}.
  \end{align*}
  Since $\lim_\lambda a_\lambda f_k = f_k$ for all $k\in I$, we have $\lim_\lambda\norm{a_\lambda f - f}_{C(G)}=0$ and hence $\lim_\lambda\norm{a_\lambda f - f}_E=0$.
\end{proof}

\begin{proposition}\label{5.20}
  An approximate unit of $\hat\rho_A(A)$ is also an approximate unit of $C^*(\hat\rho)$.
\end{proposition}

\begin{proof}
  Let $\{\alpha_\lambda\}$ be an approximate unit of $\hat\rho_A(A)$.
  Since $\hat\rho_A$ is an isomorphism between $A$ and $\hat\rho_A(A)$, letting $a_\lambda:=\hat\rho_A^{-1}(\alpha_\lambda)$ for each $\lambda$, the family $\{a_\lambda\}$ is also an approximate unit for $A$.
  
  We show the equation $\lim_\lambda\beta\alpha_\lambda=\lim_\lambda\alpha_\lambda\beta=\beta$ for all $\beta\in\hat E_1=\hat\rho_E(E)$.
  Take any $\beta\in\hat E_1$, which can be written in the form $\beta=\hat\rho_E(f)$ for some $f\in E$.
  By Fact~\ref{2.8}, $\lim_\lambda f a_\lambda = f$ holds; and by Lemma \ref{5.19}, $\lim_\lambda a_\lambda f = f$ holds.
  Since $\hat\rho_E$ is isometry, we have 
  \begin{equation*}
    \lim_\lambda\beta\alpha_\lambda
    =\lim_\lambda\hat\rho_E(f)\hat\rho_A(a_\lambda)
    =\lim_\lambda\hat\rho_E(fa_\lambda)
    =\hat\rho_E(\lim_\lambda fa_\lambda)
    =\hat\rho_E(f)
    =\beta
  \end{equation*}
  and
  \begin{equation*}
    \lim_\lambda\alpha_\lambda\beta
    =\lim_\lambda\hat\rho_A(a_\lambda)\hat\rho_E(f)
    =\lim_\lambda\hat\rho_E(a_\lambda f)
    =\hat\rho_E(\lim_\lambda a_\lambda f)
    =\hat\rho_E(f)
    =\beta.
  \end{equation*}
  Therefore, the equation $\lim_\lambda\beta\alpha_\lambda=\lim_\lambda\alpha_\lambda\beta=\beta$ holds for all $\beta\in\hat E_1$.

  The equation $\lim_\lambda(\bullet\cdot\alpha_\lambda)=\lim_\lambda(\alpha_\lambda\cdot\bullet)=\bullet$ holds on $\hat\rho_A(A)=\hat E_0$ and on $\hat E_1$; and hence on the $*$-subalgebra $B$ of $C^*(\hat\rho)$ generated by $\hat E_0\cup\hat E_1$.
  The $*$-subalgebra $B$ is dense in $C^*(\hat\rho)$.

  Take any $\beta\in C^*(\hat\rho)$.
  Let $\varepsilon$ be a positive real number.
  We can take $\beta'\in B$ satisfying $\norm{\beta'-\beta}<\varepsilon$.
  Then, we have 
  \begin{align*}
    \norm{\beta\alpha_\lambda-\beta}
    &\leq\norm{(\beta-\beta')(\alpha_\lambda-1)}+\norm{\beta'\alpha_\lambda-\beta'}\\
    &=\norm{\beta-\beta'}\bigl(\norm{\alpha_\lambda}+\norm{1}\bigr)+\norm{\beta'\alpha_\lambda-\beta'}\\
    &<2\varepsilon+\norm{\beta'\alpha_\lambda-\beta'}.
  \end{align*}
  Taking limits on $\lambda$, we get $\lim_\lambda\norm{\beta\alpha_\lambda-\beta}<2\varepsilon$.
  Since this holds for any $\varepsilon>0$, we obtain that $\lim_\lambda\norm{\beta\alpha_\lambda-\beta}=0$.
  In other words, $\lim_\lambda(\beta\alpha_\lambda)=\beta$ holds.

  Therefore, $\{\alpha_\lambda\}$ is an approximate unit for $C^*(\hat\rho)$.
\end{proof}

\begin{corollary}\label{5.21}
  $\hat D$ contains an approximate unit of $C^*(\rho)$.
\end{corollary}

\begin{proof}
  $\hat D$ includes $\hat\rho_A(A)$.
  Hence, by Proposition \ref{5.20}, an approximate unit of $\hat\rho_A(A)$ is an approximate unit of $C^*(\hat\rho)$ and is contained in $\hat D$.
\end{proof}

\subsection{Conditional expectation}

\begin{fact}[\cite{B}, II.9.7.1]\label{5.22}
  Let $A_1$, $A_2$, $B_1$ and $B_2$ be C*-algebras and $\phi\colon A_1\to A_2$ and $\psi\colon B_1\to B_2$ be completely positive contractions.
  Then, there exists a unique completely positive contraction $\phi\otimes\psi$ from $A_1\otimes_{\rm min}B_1$ to $A_2\otimes_{\rm min}B_2$ satisfying the equation $(\phi\otimes\psi)(a\otimes b)=\phi(a)\otimes\psi(b)$ for all $a\in A_1$ and $b\in B_1$.
  If $\phi$ and $\psi$ are conditional expectations, then so is $\phi\otimes\psi$.
\end{fact}

\begin{proposition}\label{5.23}
  Assume the following conditions:
  \begin{itemize}
    \item The iterated function system $\gamma$ is essentially free.
    \item The set $\gamma^n(X)$ is clopen in $X$ for all $n\in\N$.
    \item There exists a continuous map $\sigma$ from $\gamma(X)$ to $X$ such that $\sigma\circ\gamma_k=\id_X$ holds for all $k\in\{1,\ldots,N\}$.
  \end{itemize}
  Then, there exists a faithful conditional expectation of $C^*(\hat\rho)$ onto $\hat D$.
\end{proposition}

\begin{proof}
  In this proof, for a set $X$, let $\id_X$ denote the identity map from $M_X$ to itself.
  We note that $\id_X$ is a conditional expectation of $M_X$ onto $M_X$.
  By Fact \ref{2.5}, for a set $X$, there exists a faithful conditional expectation $\delta_X$ of $M_X$ onto $D_X$ satisfying the equation $\delta_X(\alpha)\e_x=\e_x\inn{\e_x}{\alpha\e_x}$ for all $\alpha\in M_X$ and $x\in X$.

  We consider $\id_{\X}$, $\id_\Z$, $\delta_{\X}$, $\delta_\Z$ and $\delta_{\X\times\Z}$.
  By Fact \ref{5.22}, the following three maps are conditional expectations:
  \begin{align*}
    \delta_{\X}\otimes\id_\Z&\colon M_{\X}\otimes_{\rm min}M_\Z\twoheadrightarrow D_{\X}\otimes_{\rm min}M_\Z,\\
    \id_{\X}\otimes\delta_\Z&\colon M_{\X}\otimes_{\rm min}M_\Z\twoheadrightarrow M_{\X}\otimes_{\rm min}D_\Z,\\
    \delta_{\X}\otimes\delta_\Z&\colon M_{\X}\otimes_{\rm min}M_\Z\twoheadrightarrow D_{\X}\otimes_{\rm min}D_\Z.
  \end{align*}
  We note that $M_{\X}\otimes_{\rm min}M_\Z\cong M_{\X\times\Z}$ and $D_{\X}\otimes_{\rm min}D_\Z\cong D_{\X\times\Z}$.
  Under these isomorphisms, $\delta_{\X}\otimes\delta_\Z$ can be identified with $\delta_{\X\times\Z}$; and hence is faithful.
  Since the composition of $\id_{\X}\otimes\delta_\Z$ and $\delta_{\X}\otimes\id_\Z$ coincides with $\delta_{\X}\otimes\delta_\Z$, $\delta_{\X}\otimes\id_\Z$ is faithful.

  Since $C^*(\Z)$ is a closed $*$-subalgebra of $M_\Z$, $C^*(\hat\rho)$ is a closed $*$-subalgebra of $M_{\X}\otimes_{\rm min}M_\Z$.
  Let $P$ be the restriction of $\delta_{\X}\otimes\id_\Z$ to $C^*(\hat\rho)$.
  Since $P$ is the composition of $\delta_{\X}\otimes\id_\Z$ and the injective $*$-homomorphism $C^*(\hat\rho)\hookrightarrow M_{\X}\otimes M_\Z$, $P$ is a faithful completely positive contraction.
  Since $\hat D$ is included in $D_{\X}\otimes M_\Z$, we have 
  \begin{gather*}
    P(\eta)=(\delta_{\X}\otimes\id_\Z)(\eta)=\eta,\\
    P(\xi\eta)=(\delta_{\X}\otimes\id_\Z)(\xi\eta)=(\delta_{\X}\otimes\id_\Z)(\xi)\eta=P(\xi)\eta,\\
    P(\eta\xi)=(\delta_{\X}\otimes\id_\Z)(\eta\xi)=\eta(\delta_{\X}\otimes\id_\Z)(\xi)=\eta P(\xi).
  \end{gather*}
  for all $\xi\in C^*(\hat\rho)$ and $\eta\in\hat D$.
  The range of $P$ includes $\hat D=P(\hat D)$.

  It remains to prove that the range of $P$ is contained in $\hat D$.
  
  By Lemma~\ref{5.5.11}, for any $k\in\Z\setminus\{0\}$ and $\alpha\in B_k$, we have $\delta_{\X}(\alpha)=0$.
  So we obtain 
  \begin{equation*}
    P(\alpha\otimes\upsilon_k)
    = \delta_{\X}(\alpha)\otimes\upsilon_k
    = 0\otimes\upsilon_k
    = 0
    \in\hat D.
  \end{equation*}
  By Lemma~\ref{5.5.6}, for any $\alpha\in B_0$, the function $x\mapsto\inn{\e_x}{\alpha\e_x}$ on $\X$ has a unique continuous extension to a function $a$ on $X$ and we have $\delta_{\X}(\alpha)=\rho_A(a)$.
  So we obtain 
  \begin{equation*}
    P(\alpha\otimes\upsilon_0)
    = \delta_{\X}(\alpha)\otimes\upsilon_0
    = \rho_A(a)\otimes\upsilon_0
    = \hat\rho_A(a)
    \in\hat D.
  \end{equation*}
  Thus, the range of $P$ is contained in $\hat D$.
\end{proof}

\subsection{Regularity}

In this subsection, we assume the following conditions: 
\begin{itemize}
  \item The space $X$ is first-countable.
  \item The iterated function system $\gamma$ is essentially free.
  \item The set $\gamma^n(X)$ is clopen in $X$ for all $n\in\N$.
  \item There exists a continuous map $\sigma$ from $\gamma(X)$ to $X$ such that $\sigma\circ\gamma_k=\id_X$ holds for all $k\in\{1,\ldots,N\}$.
\end{itemize}

\begin{definition}[\cite{R}, Definition~4.5]\label{5.9.1}
  Let $A$ be a C*-subalgebra of a C*-algebra $B$.
  \begin{enumerate}
    \item Its \emph{normalizer} is the set 
    \begin{equation*}
      N(A,B) := \set{b\in B}{bAb^*\subseteq A\ \text{and}\ b^*Ab\subseteq A}.
    \end{equation*}
    \item One says that $A$ is \emph{regular} if its normalizer $N(A,B)$ generates $B$ as a C*-algebra.
  \end{enumerate}
\end{definition}

\begin{definition}\label{5.9.2}
  An iterated function system $\gamma = \{\gamma_1,\ldots,\gamma_N\}$ satisfies \emph{the graph separation condition} if $\gamma_j(x)\neq\gamma_k(x)$ holds for all distinct indices $j$ and $k$ and all points $x\in X$.
\end{definition}

\begin{definition}\label{5.9.3}
  Let $k\in\N_0$ and $\alpha\in M_\X$.
  The function on the set 
  \begin{equation*}
    \set{(x,\gamma_{\boldsymbol j}(x))}{x\in\X,\ \boldsymbol j\in I^k} 
  \end{equation*}
  which sends $(x,y)$ to the $(x,y)$-entry of $\alpha$ is called \emph{the $k$-th entry function} for $\alpha$.
  Also, its continuous extension to 
  \begin{equation*}
    \set{(x,\gamma_{\boldsymbol j}(x))}{x\in X,\ \boldsymbol j\in I^k}, 
  \end{equation*}
  if exists, is called \emph{the $k$-th extended entry function}.
\end{definition}

\begin{proposition}\label{5.9.4}
  Let $k\in\N_0$.
  Every $\X$-squared matrix in $C^*(\rho)$ admits the $k$-th extended entry function.
\end{proposition}

\begin{proof}
  For $n\in\N_0$, define a homeomorphism $\boldsymbol\sigma_n\colon\gamma^n(X)\to X_n$ by 
  \begin{equation*}
    \boldsymbol\sigma_n(x)
    := (x, \sigma(x), \sigma^2(x), \ldots, \sigma^n(x)).
  \end{equation*}

  Let $m,n\in\N_0$ with $m - n = k$.
  Let $\alpha\in E_{m,n}$.
  Let $f$ be the $k$-th entry function for $\alpha$.
  There exists $g\in C(X_{m,n})$ such that $\alpha = \rho_{m,n}(g)$.
  The equation 
  \begin{equation*}
    f(x,y)
    = \begin{cases}
      g(\boldsymbol\sigma_m(x),\boldsymbol\sigma_n(y)) & \text{if}\ x\in\gamma^m(X)\ \text{and}\ y\in\gamma^n(X),\\
      0 & \text{otherwise}.
    \end{cases}
  \end{equation*}
  holds for all $(x,y)\in\set{(z,\gamma_{\boldsymbol j}(z))}{z\in\X,\ \boldsymbol j\in I^k}$.
  Since $\gamma^m(X)$ and $\gamma^n(X)$ are clopen sets, the right-hand side function is continuous on $\set{(z,\gamma_{\boldsymbol j}(z))}{z\in X,\ \boldsymbol j\in I^k}$.
  Hence, $f$ has a continuous extension to the set $\set{(z,\gamma_{\boldsymbol j}(z))}{z\in X,\ \boldsymbol j\in I^k}$.
  Therefore, for all $m,n\in\N_0$ with $m-n=k$, every $\X$-squared matrix in $E_{m,n}$ admits the $k$-th extended entry function.

  Let $l\in\Z$ with $l\neq k$.
  Let $\alpha\in B_l$.
  By Corollary~\ref{5.4.6} and Corollary~\ref{5.4.8}, the $k$-th entry function for $\alpha$ is the constant function $0$.
  Hence, it also has a continuous extension to the set $\set{(z,\gamma_{\boldsymbol j}(z))}{z\in X,\ \boldsymbol j\in I^k}$.
  Therefore, for all $l\in\Z$ with $l\neq k$, every $\X$-squared matrix in $B_l$ admits the $k$-th extended entry function.

  Finally, note that the operator mapping an $\X$-squared matrix to its $k$-th entry function is linear and bounded.
  Thus, by Corollary~\ref{5.4.3}, every $\X$-squared matrix in $C^*(\rho)$ admits the $k$-th extended entry function.
\end{proof}

\begin{lemma}\label{5.9.5}
  $N(D_0, C^*(\rho)) = C^*(\rho) \cap QM_\X$.
\end{lemma}

\begin{proof}
  First, we show the inclusion $\supseteq$.
  Take any $\alpha\in C^*(\rho)\cap QM_\X$ and $\beta\in D_0$.
  Consider any two points $w,z\in\X$ such that $w\rel{\alpha^*\beta\alpha}z$.
  By Fact~\ref{2.4}, there exist points $x,y\in\X$ such that $x\rel\alpha w$, $x\rel\beta y$, and $y\rel\alpha z$.
  Since $\beta$ is diagonal, $x\rel\beta y$ implies that $x=y$.
  Since $\alpha$ is quasi-monomial, $x\rel\alpha w$ and $x=y\rel\alpha z$ imply $w=z$.
  Therefore, the relation $\rel{\alpha^*\beta\alpha}$ is diagonal, and $\alpha^*\beta\alpha$ is a diagonal matrix.
  Moreover, this matrix belongs to $C^*(\rho)$.
  By Theorem~\ref{5.5.5}, we have $\alpha^*\beta\alpha\in D_0$.
  Similarly, $\alpha\beta\alpha^* \in D_0$.
  Since this holds for any $\beta\in D_0$, we conclude that $\alpha\in N(D_0, C^*(\rho))$.
  Thus, $N(D_0, C^*(\rho)) \supseteq C^*(\rho) \cap QM_\X$.

  Next, we show the reverse inclusion $\subseteq$.
  Take any $\alpha\in C^*(\rho)\setminus QM_\X$.
  There exist three points $u,v,w\in\X$ satisfying one of the following conditions:
  \begin{center}
    (1)\ $u\rel\alpha v$, $u\rel\alpha w$, \text{and} $v\neq w$\quad\text{or}\quad
    (2)\ $u\rel\alpha w$, $v\rel\alpha w$, \text{and} $u\neq v$.
  \end{center}
  Without loss of generality, we may assume (1) holds.
  For $x\in X$, define $f(x)$ and $g(x)$ on $X$ as the $(x,v)$-entry and $(x,w)$-entry of $\alpha$, respectively.
  These $f$ and $g$ are $L^2$ functions on the measure space $\X$ with counting measure and do not vanish at $u$.
  Since $X$ is first-countable and compact, we can take a decreasing sequence $U_1\supseteq U_2\supseteq\cdots$ of open neighborhoods of $u$ satisfying that $\bigcap_{k\in\N}U_k = \{u\}$.
  For each $k$, we can take a continuous function $h_k$ with values in $[0,1]$ on $X$ which takes the value $1$ at $u$ and is supported on $U_k$.
  The sequence $\{\bar f h_k g\}_{k\in\N}$ is dominated by $\abs{fg}$ and converges pointwise to the function that takes the value $\overline{f(u)}g(u)$ at $u$ and vanishes otherwise.
  By the Dominated Convergence Theorem, we obtain that 
  \begin{equation*}
    \lim_{k\to\infty}\sum_{x\in\X}\overline{f(x)}h_k(x)g(x)
    = \overline{f(u)}g(u)
    \neq 0.
  \end{equation*}
  Therefore, letting $h:=h_K$ for some $K$, we have $\sum_{x\in\X}\overline{f(x)}h(x)g(x)\neq 0$, and 
  \begin{align*}
    \inn{\e_v}{\alpha^*\rho_A(h)\alpha\e_w}
    &= \inn{\e_v}{\alpha^*\sum_{x\in\X}\e_x\inn{\e_x}{\rho_A(h)\sum_{y\in\X}\e_y\inn{\e_y}{\alpha\e_w}}}\\
    &= \sum_{x\in\X}\sum_{y\in\X}\overline{\inn{\e_x}{\alpha\e_v}}\inn{\e_x}{\rho_A(h)\e_y}\inn{\e_y}{\alpha\e_w}\\
    &= \sum_{x\in\X}\overline{\inn{\e_x}{\alpha\e_v}}\inn{\e_x}{\rho_A(h)\e_x}\inn{\e_x}{\alpha\e_w}\\
    &= \sum_{x\in\X}\overline{f(x)}h(x)g(x)\\
    &\neq 0.
  \end{align*}
  Therefore, the $\X$-squared matrix $\alpha^*\rho_A(h)\alpha$ is not diagonal, and hence it does not belong to $\rho_A(A)$.
  Hence, $\alpha^*\rho_A(A)\alpha\not\subseteq\rho_A(A)$ implying $\alpha\notin N(D_0,C^*(\rho))$.

  Therefore, we have 
  \begin{equation*}
    C^*(\rho)\setminus QM_\X \subseteq C^*(\rho)\setminus N(D_0,C^*(\rho))\ \text{and}\ N(D_0,C^*(\rho))\subseteq C^*(\rho)\cap QM_\X.
    \qedhere
  \end{equation*}
\end{proof}

\begin{lemma}\label{5.9.6}
  If $\gamma$ does not satisfy the graph separation condition, $D_0$ is not regular as a C*-subalgebra of $C^*(\rho)$.
\end{lemma}

\begin{proof}
  Assume that $\gamma$ does not satisfy the graph separation condition.
  Then, there exist distinct indices $j$ and $k$ and a point $x\in X$ such that $\gamma_j(x) = \gamma_k(x) =: y$.
  The pair $(x,y)$ is in the set $\set{(x,\gamma_i(x))}{x\in X,\ i\in I}$, which is the domain of the 1st extended entry functions.

  Let $\alpha\in C^*(\rho)$ and let $f$ be the 1st extended entry function for $\alpha$.
  Assume that $f(x,y)\neq 0$.
  The open set 
  \begin{equation*}
    V := \set{z\in X}{f(z,\gamma_j(z))\neq 0}\cap\set{z\in X}{f(z,\gamma_k(z))\neq 0}
  \end{equation*}
  contains the point $x$, so $V$ is non-empty.
  Since $\X$ is dense, $\X$ intersects $V$.
  Take a point $w$ in the intersection $\X\cap V$.
  We have 
  \begin{equation*}
    (\text{the $(w,\gamma_j(w))$-entry of $\alpha$})
    = f(w,\gamma_j(w))
    \neq 0
  \end{equation*}
  and hence $w\rel\alpha\gamma_j(w)$.
  Similarly, we have $w\rel\alpha\gamma_k(w)$.
  Since the images $\gamma_j(X)$ and $\gamma_k(X)$ are disjoint, we have $\gamma_j(w)\neq\gamma_k(w)$.
  Therefore, $\alpha$ is not quasi-monomial.
  By Lemma~\ref{5.9.5}, $\alpha\notin N(D_0,C^*(\rho))$ holds.
  
  This shows that for every matrix in $N(D_0,C^*(\rho))$, its 1st extended entry function vanishes at $(x,y)$.
  This property extends to the C*-subalgebra of $C^*(\rho)$ generated by $N(D_0, C^*(\rho))$.

  Now, consider the constant function $g$ on $X_{1,0}$ defined by $g = 1$, which belongs to $C(X_{1,0})$.
  The 1st entry function for $\rho_{1,0}(g)$ takes the constant value $1$, so its 1st extended entry function at $(x,y)$ is also $1$.
  Thus, $\rho_{1,0}(g)$ is not in the C*-subalgebra of $C^*(\rho)$ generated by $N(D_0,C^*(\rho))$ but is in $C^*(\rho)$.
  Therefore, $D_0$ is not regular as a C*-subalgebra of $C^*(\rho)$.
\end{proof}

\begin{lemma}\label{5.9.7}
  If $\gamma$ satisfies the graph separation condition, then $D_0$ is regular as a C*-subalgebra of $C^*(\rho)$.
\end{lemma}

\begin{proof}
  In this proof, let $C^*(N)$ denote the C*-subalgebra of $C^*(\rho)$ generated by $N(D_0,C^*(\rho))$.
  For any $a\in A = C(X)$, $\rho_A(a)$ is diagonal and hence quasi-monomial.
  By Lemma~\ref{5.9.5}, $\rho_A(A)$ is contained in $N(D_0,C^*(\rho))$.

  Let $G$ denote the graph of $\gamma$.
  Let $f\in E = C(G)$.
  For each $k\in\{1,\ldots,N\}$, let $f_k$ be the function on $G$ that coincides with $f$ on the graph of $\gamma_k$ and is zero elsewhere.
  Since $\gamma$ satisfies the graph separation condition, the graphs of $\gamma_1,\ldots,\gamma_N$ are pairwise disjoint.
  Hence, the functions $f_k$ are continuous on the whole of $G$, and we have $f = f_1 + \cdots + f_N$.
  Since the functions $f_k$ take non-zero values only on the graph of $\gamma_k$, the matrices $\rho_E(f_k)$ are quasi-monomial and, by Lemma~\ref{5.9.5}, belong to $N(D_0,C^*(\rho))$.
  Therefore, the matrix $\rho_E(f) = \rho_E(f_1) + \cdots + \rho_E(f_N)$ belongs to $C^*(N)$.
  
  By the above discussion, $\rho_A(A)\cup\rho_E(E)$ is contained in $C^*(N)$.
  Since $C^*(\rho)$ is generated by $\rho_A(A)\cup\rho_E(E)$, we conclude that $C^*(N) = C^*(\rho)$, and hence $D_0$ is regular as a C*-subalgebra of $C^*(\rho)$.
\end{proof}

\begin{theorem}\label{5.9.8}
  The following are equivalent:
  \begin{enumerate}
    \item $\gamma$ satisfies the graph separation condition.
    \item $D_0$ is regular as a C*-subalgebra of $C^*(\rho)$.
    \item $\hat D$ is regular as a C*-subalgebra of $C^*(\hat\rho)$.
  \end{enumerate}
\end{theorem}

\begin{proof}
  The equivalence between (1) and (2) immediately follows from Lemma~\ref{5.9.6} and Lemma~\ref{5.9.7}.
  Since $\gamma$ is essentially free, by Lemma~\ref{5.5.10}, if necessary, $\X$ can be replaceed to ensure the condition (V).
  Then, by Theorem~\ref{5.4.12}, $\rho$ becomes universal.
  Since $\hat\rho$ is also universal (Theorem~\ref{5.2}), $\rho$ is isomorphic to $\hat\rho$.
  By Theorem~\ref{5.5.5}, $\hat D = D_0\otimes\upsilon_0$ holds.
  Therefore, (2) and (3) are also equivalent.
\end{proof}

\subsection{Cartan subalgebra}

\begin{definition}[\cite{R}, Definition~5.1]\label{5.10.1}
  We shall say that an abelian C*-subalgebra $A$ of a C*-algebra $B$ is a \emph{Cartan subalgebra} if 
  \begin{enumerate}
    \item $A$ contains an approximate unit of $B$;
    \item $A$ is maximal abelian;
    \item $A$ is regular;
    \item there exists a faithful conditional expectation $P$ of $B$ onto $A$.
  \end{enumerate}
\end{definition}

\begin{theorem}\label{5.10.2}
  The following are equivalent:
  \begin{enumerate}
    \item $\gamma$ satisfies the graph separation condition.
    \item $\hat D$ is a Cartan subalgebra of $C^*(\hat\rho)$.
  \end{enumerate}
\end{theorem}

\begin{proof}
  This immediately follows from Proposition~\ref{5.20}, Theorem~\ref{5.5.1}, Theorem~\ref{5.9.8}, and Proposition~\ref{5.23}.
\end{proof}

\section{Examples}

\subsection{Masa}

\begin{lemma}\label{6.1.1}
  Let $X$ be a compact metric space and $\gamma = \{\gamma_1,\ldots,\gamma_N\}$ be an iterated function system consisting of embeddings.
  If $\gamma$ satisfies the open set condition, then there exist subsets $\X$ and $U$ of $X$ satisfying the following conditions:
  \begin{enumerate}
    \item The set $\X$ is dense in $X$ and completely invariant under $\gamma$.
    \item The images $\gamma_1(\X),\ldots,\gamma_N(\X)$ are pairwise disjoint.
    \item The restriction $\gamma_k|_\X$ is injective for all $k=1,\ldots,N$.
    \item The maps $\gamma_1,\ldots,\gamma_N$ are closed embeddings.
    \item The set $U$ is a dense open set in $X$ such that the images $\gamma_1(U),\ldots,\gamma_N(U)$ are pairwise disjoint open set in $X$ and contained in $U$.
    \item The set $\X$ is contained in $U$.
    \item The set $\X$ is comeager in $X$.
  \end{enumerate}
\end{lemma}

\begin{proof}
  \newcommand\Y{\mathbb Y}
  The condition (4) always holds since $X$ is compact.
  Since $\gamma$ satisfies the open set condition, we can take a subset $U$ of $X$ satisfying the condition (5).
  The complement of $U$, denoted by $\complement U$, is a nowhere dense closed set.
  Let $\Y$ be the completely $\gamma$-invariant subset of $X$ generated by $\complement U$ and $\X$ be the complement of $\Y$.
  By Lemma~\ref{3.11}, the set $\Y$ is meager in $X$ and hence the condition (7) holds.
  By the definition of $\Y$, we have $\complement U\subseteq\Y$.
  By taking the complement of both sides, we obtain $U \supseteq \X$, and thus condition (6) holds.
  By the Baire category theorem, the condition (7) implies that $X$ is dense in $X$.
  The set $\Y$ is completely invariant under $\gamma$ by its definition, and hence so is its complement $\X$.
  So the condition (1) holds.
  Since the images $\gamma_1(U),\ldots,\gamma_N(U)$ are pairwise disjoint (the condition (6)), the condition (6) implies the condition (2).
  The condition (4) implies the condition (3).
\end{proof}

\begin{theorem}\label{6.1.2}
  Let $X$ be a compact metric space and $\gamma$ be an iterated function system consisting of embeddings.
  If $\gamma$ is not essentially free, then $C(X)$ is not a masa of $\O_\gamma$.
\end{theorem}

\begin{proof}
  This is immediately follows from Lemma~\ref{6.1.1} and Theorem~\ref{5.5.5}.
\end{proof}

\begin{example}\label{6.1.3}
  Let $X$ be a compact metric space and $\gamma = \{\gamma_1,\ldots,\gamma_N\}$ be an iterated function system consisting of embeddings.
  Suppose that the fixed point set of some $\gamma_k$ has non-empty interior.
  Then, $\gamma$ is clearly not essentially free and hence, by Theorem~\ref{6.1.2}, $C(X)$ is not a masa of $\O_\gamma$.
\end{example}

\begin{theorem}\label{6.1.4}
  Let $X$ be a compact metric space and $\gamma$ be an iterated function system consisting of embeddings.
  Assume the following conditions:
  \begin{itemize}
    \item the iterated function system $\gamma$ is essentially free; 
    \item the set $\gamma^n(X)$ is clopen in $X$ for all $n\in\N$; 
    \item there exists a continuous map $\sigma$ from $\gamma(X)$ to $X$ such that $\sigma\circ\gamma_k = \id_X$ holds for all $k\in\{1,\ldots,N\}$.
  \end{itemize}
  Then, $C(X)$ is a masa of $\O_\gamma$.
\end{theorem}

\begin{proof}
  This is immediately follows from Lemma~\ref{6.1.1}, Theorem~\ref{5.5.5}, and Theorem~\ref{5.5.7}.
\end{proof}

\begin{lemma}\label{6.1.5}
  Let $X$ be a metric space with no isolated points and $\gamma$ be an iterated function system.
  If $\gamma$ consists of proper contractions, then $\gamma$ is essentially free.
\end{lemma}

\begin{proof}
  We prove the contrapositive.
  Assume that $\gamma$ is not essentially free.
  Then, by Lemma~\ref{5.5.3}, there exists a non-empty finite sequence $\boldsymbol k$ of indices $\{1,\ldots,N\}$ such that the fixed point set of $\gamma_{\boldsymbol k}$ has non-empty interior.
  Since $X$ has no isolated points, we can take two distinct fixed points $x$ and $y$ of $\gamma_{\boldsymbol k}$.
  We have $d(\gamma_{\boldsymbol k}(x),\gamma_{\boldsymbol k}(y)) = d(x,y)$ and hence $\gamma_{\boldsymbol k}$ is not a proper contraction.
\end{proof}

\begin{example}\label{6.1.6}
  Let $X := [0,1]$.
  Define $\gamma_1$ and $\gamma_2$ by 
  \begin{equation*}
    \gamma_1(x) := \dfrac 1 2 x,\quad
    \gamma_2(x) := - \dfrac 1 2 x + 1
  \end{equation*}
  for $x\in X$.
  Then, $\gamma = \{\gamma_1, \gamma_2\}$ becomes an iterated function system on $X$ consisting of proper contractions.
  By Lemma~\ref{6.1.5}, $\gamma$ is essentially free.
  For all $n\in\N$, the set $\gamma^n(X) = X$ is clopen.
  Define a map $\sigma$ from $\gamma(X) = X$ to $X$ by 
  \begin{equation*}
    \sigma(x) := \begin{cases}
      2x & \text{if}\ 0\leq x\leq 1/2;\\
      -2x+2 & \text{if}\ 1/2\leq x\leq 1.
    \end{cases}
  \end{equation*}
  Then, $\sigma\circ\gamma_1 = \sigma\circ\gamma_2 = \id_X$ holds.
  Therefore, by Theorem~\ref{6.1.4}, $C(X)$ is a masa of $\O_\gamma$.
\end{example}

\begin{example}\label{6.1.7}
  Let $X$ be the Cantor ternary set, which is contained in $[0,1]$.
  Define $\gamma_1$ and $\gamma_2$ by 
  \begin{equation*}
    \gamma_1(x) := \dfrac 1 3 x,\quad
    \gamma_2(x) := \dfrac 1 3 x + \dfrac 2 3
  \end{equation*}
  for $x\in X$.
  Then, $\gamma = \{\gamma_1, \gamma_2\}$ becomes an iterated function system on $X$ consisting of proper contractions.
  By Lemma~\ref{6.1.5}, $\gamma$ is essentially free.
  For all $n\in\N$, the set $\gamma^n(X) = X$ is clopen.
  Define a map $\sigma$ from $\gamma(X) = X$ to $X$ by 
  \begin{equation*}
    \sigma(x) := \begin{cases}
      3x & \text{if}\ 0\leq x\leq 1/3;\\
      3x-2 & \text{if}\ 2/3\leq x\leq 1.
    \end{cases}
  \end{equation*}
  Then, $\sigma\circ\gamma_1 = \sigma\circ\gamma_2 = \id_X$ holds.
  Therefore, by Theorem~\ref{6.1.4}, $C(X)$ is a masa of $\O_\gamma$.
\end{example}

\subsection{Cartan subalgebra}

\begin{theorem}\label{6.2.1}
  Let $X$ be a compact metric space and $\gamma$ be an iterated function system consisting of embeddings.
  Assume the following conditions:
  \begin{itemize}
    \item the iterated function system $\gamma$ is essentially free; 
    \item the set $\gamma^n(X)$ is clopen in $X$ for all $n\in\N$; 
    \item there exists a continuous map $\sigma$ from $\gamma(X)$ to $X$ such that $\sigma\circ\gamma_k = \id_X$ holds for all $k\in\{1,\ldots,N\}$.
  \end{itemize}
  Then, the following are equivalent:
  \begin{enumerate}
    \item $\gamma$ satisfies the graph separation condition, i.e., for all distinct indices $j$ and $k$ and all points $x \in X$, one has $\gamma_j(x) \neq \gamma_k(x)$.
    \item $C(X)$ is a Cartan subalgebra of $\O_\gamma$.
  \end{enumerate}
\end{theorem}

\begin{proof}
  This is immediately follows from Lemma~\ref{6.1.1}, Theorem~\ref{5.9.8}, and Theorem~\ref{5.10.2}.
\end{proof}

\begin{example}\label{6.2.2}
  Consider the setting in Example~\ref{6.1.6}.
  Since $\gamma_1(1/2) = \gamma_2(1/2)$ holds, $\gamma$ does not satisfy the graph separation condition.
  Therefore, by Theorem~\ref{6.2.1}, the masa $C(X)$ is not a Cartan subalgebra of $\O_\gamma$.
\end{example}

\begin{example}\label{6.2.3}
  Consider the setting in Example~\ref{6.1.7}.
  The range of $\gamma_1$ is $X\cap [0,1/3]$ and the range of $\gamma_2$ is $X\cap [2/3,1]$.
  So the range of them are disjoint, and hence $\gamma$ satisfies the graph separation condition.
  Therefore, by Theorem~\ref{6.2.1}, the masa $C(X)$ is a Cartan subalgebra of $\O_\gamma$.
\end{example}


\begin{bibdiv}
  \begin{biblist}
  \bib{B}{book}{
     author={Blackadar, B.},
     title={Operator algebras},
     series={Encyclopaedia of Mathematical Sciences},
     volume={122},
     note={Theory of $C^*$-algebras and von Neumann algebras;
     Operator Algebras and Non-commutative Geometry, III},
     publisher={Springer-Verlag, Berlin},
     date={2006},
     pages={xx+517},
     isbn={978-3-540-28486-4},
     isbn={3-540-28486-9},
  }
  \bib{DM}{article}{
     author={Deaconu, V.},
     author={Muhly, P. S.},
     title={$C^*$-algebras associated with branched coverings},
     journal={Proc. Amer. Math. Soc.},
     volume={129},
     date={2001},
     pages={1077--1086},
     issn={0002-9939},
  }
  \bib{H}{article}{
   author={Hutchinson, J. E.},
   title={Fractals and self-similarity},
   journal={Indiana Univ. Math. J.},
   volume={30},
   date={1981},
   pages={713--747},
   issn={0022-2518},
  }
  \bib{I}{article}{
    author={Ito, K.},
    title={Cartan subalgebras of C*-algebras associated with complex dynamical systems},
    journal={J. Operator Theory},
    note={Accepted in 2024, in press},
    eprint={arXiv:2303.14860},
    archivePrefix={arXiv},
  }
  \bib{KW05}{article}{
     author={Kajiwara, T.},
     author={Watatani, Y.},
     title={$C^\ast$-algebras associated with complex dynamical systems},
     journal={Indiana Univ. Math. J.},
     volume={54},
     date={2005},
     pages={755--778},
     issn={0022-2518},
  }
  \bib{KW06}{article}{
     author={Kajiwara, T.},
     author={Watatani, Y.},
     title={$C^\ast$-algebras associated with self-similar sets},
     journal={J. Operator Theory},
     volume={56},
     date={2006},
     pages={225--247},
     issn={0379-4024},
  }
  \bib{KW17}{article}{
     author={Kajiwara, T.},
     author={Watatani, Y.},
     title={Maximal abelian subalgebras of $C^\ast$-algebras associated with
     complex dynamical systems and self-similar maps},
     journal={J. Math. Anal. Appl.},
     volume={455},
     date={2017},
     pages={1383--1400},
     issn={0022-247X},
  }
  \bib{K}{article}{
     author={Katsura, T.},
     title={On $C^*$-algebras associated with $C^*$-correspondences},
     journal={J. Funct. Anal.},
     volume={217},
     date={2004},
     pages={366--401},
     issn={0022-1236},
  }
  \bib{Ku}{article}{
    author={Kumjian, A.},
    title={On $C^\ast$-diagonals},
    journal={Canad. J. Math.},
    volume={38},
    date={1986},
    pages={969--1008},
    issn={0008-414X},
  }
   \bib{P}{article}{
     author={Pimsner, M. V.},
     title={A class of $C^*$-algebras generalizing both Cuntz--Krieger algebras
     and crossed products by ${\bf Z}$},
     conference={
        title={Free probability theory},
        address={Waterloo, ON},
        date={1995},
     },
     book={
        series={Fields Inst. Commun.},
        volume={12},
        publisher={Amer. Math. Soc., Providence, RI},
     },
     date={1997},
     pages={189--212},
  }
  \bib{R1980}{book}{
    author={Renault, J.},
    title={A groupoid approach to $C\sp{\ast} $-algebras},
    series={Lecture Notes in Mathematics},
    volume={793},
    publisher={Springer, Berlin},
    date={1980},
    pages={ii+160},
    isbn={3-540-09977-8},
  }
   \bib{R}{article}{
     author={Renault, J.},
     title={Cartan subalgebras in $C^*$-algebras},
     journal={Irish Math. Soc. Bull.},
     number={61},
     date={2008},
     pages={29--63},
  }
  \end{biblist}
\end{bibdiv}
\end{document}